\documentclass[a4paper,reqno,11pt]{article}
\usepackage{amssymb,amsfonts,amsmath,amsbsy,theorem,amscd}

\usepackage{latexsym}

\begingroup
\newtheorem{theorem}{Theorem}[section]
\newtheorem{lemma}[theorem]{Lemma}
\newtheorem{proposition}[theorem]{Proposition}
\newtheorem{corollary}[theorem]{Corollary}
\endgroup

\begingroup

\newtheorem{remark}[theorem]{Remark}

\endgroup

\mathchardef\emptyset="001F

\newenvironment{proof}{{\bf
Proof:\,}}{\hspace*{\fill}\rule{1.2ex}{1.2ex}\\ }

\numberwithin{equation}{section}

\newcommand{\R}{\mathbb R}
\newcommand{\N}{\mathbb N}
\newcommand{\Z}{\mathbb Z}

\newcommand{\mtwo}{{\mathbb M}^{2{\times}2}}

\newcommand{\dist}{{\rm dist}}
\newcommand{\sym}{{\rm sym}\,}
\newcommand{\skw}{{\rm skew}\,}
\renewcommand{\div}{{\rm div}}

\newcommand{\ol}{\overline}

\newcommand{\eps}{\varepsilon}

\newcommand{\Id}{\operatorname{Id}}
\newcommand{\sd}{\, d}
\newcommand{\weight}[1]{\langle #1\rangle}

\newcommand{\fint}{\int}
\newcommand{\notshow}[1]{}
\newcommand{\Div}{\div}

\newenvironment{proof*}[1]{{\bf Proof
#1:}}{\hspace*{\fill}\rule{1.2ex}{1.2ex}\\ } 

\begin{document}
\begin{titlepage}
\title{Large Time Existence for Thin Vibrating Plates}
\author{H. Abels, M.G.\ Mora, and S.\ M\"uller}

\end{titlepage}
\maketitle
\begin{abstract}
We construct strong solutions for a nonlinear wave equation for a thin vibrating plate described by nonlinear elastodynamics. For sufficiently small thickness we obtain existence of strong solutions for large times under appropriate scaling of the initial values such that the limit system as $h\to 0$ is either the nonlinear von K\'arm\'an plate equation or the linear fourth order Germain-Lagrange equation. In the case of the linear Germain-Lagrange equation we even obtain a convergence rate of the three-dimensional solution to the solution of the two-dimensional linear plate equation.
\end{abstract}
\noindent{\bf Key words:} Wave equation, plate theory, von K\'arm\'an, nonlinear elasticity, dimension reduction, singular perturbation

\noindent{\bf AMS-Classification:} Primary: 74B20, 
Secondary:
35L20, 
35L70,  
74K20 
\section{Introduction}

In the present contribution we study the nonlinear wave equation for a thin vibrating plate (or rod if $d=2$). The plate is assumed to be of small but positive thickness $h>0$ and satisfies the equations of three-dimensional nonlinear elastodynamics.

In order to explain the result and the model under consideration, let us start by recalling some facts and results for the corresponding variational problems, see \cite{FrieseckeJamesMueller} for further details. We consider the elastic energy
\begin{equation*}
  \tilde{E}^h(z)= \frac1h \int_{\Omega_h} \left(W(\nabla z(x)) - f^h\cdot (z(x)-x) \right)\sd x,
\end{equation*}
where $\Omega_h= \Omega'\times (-\frac{h}2,\frac{h}2)$ is the reference configuration of the thin plate, $\Omega'\subset \R^{d-1}$, $d=2,3$, is a suitable bounded domain, and $z\colon \Omega_h \to \R^d$ is the deformation of the plate. For simplicity, we will restrict ourselves to the case $d=3$ in this introduction. Rescaling $\Omega_h$ to $\Omega= \Omega'\times (-\frac{1}2,\frac{1}2)$, we obtain the rescaled energy
\begin{equation*}
  E^h(y)=  \int_{\Omega} \left(W(\nabla_h y(x)) - f^h\cdot \left(y(x)-
      \begin{pmatrix}
        x_1\\ x_2\\ h x_3
      \end{pmatrix}
\right) \right)\sd x,
\end{equation*}
where $y(x)= z(x',hx_3)$ with $x'=(x_1,x_2)$ and $\nabla_h = (\partial_{x_1},\partial_{x_2},\frac1h \partial_{x_3})$. The limit as $h\to 0$ depends on the asymptotic behaviour of $f^h$. More precisely, let $f^h$ be of order $h^\alpha$. If $\alpha=2$, then the energy $E^h$ is of order $h^\beta$ with $\beta=2$. The rescaled energy  $\frac1{h^2}E^h$ converges as $h\to 0$ to the elastic energy from the geometrically fully nonlinear Kirchhoff theory in the sense of $\Gamma$-convergence. To the authors' knowledge there are no results on existence of solutions for the corresponding dynamic wave equation or on regularity of non-minimizing equilibria. Indeed even the precise definition of equilibrium is not completely clear since the isometry constraint $\nabla \ol{y}^T\nabla \ol{y}= \Id $ for the limit map $\ol{y}\colon \Omega'\to \R^3$ makes the problem very rigid; see Hornung~\cite{Hornung6,Hornung7} for recent progress. If $\alpha>2$ and $\beta= 2\alpha-2$, then the limit energy can be described as
\begin{equation*}
  \frac{\Lambda_\alpha}2 \int_{\Omega'} Q_2\left(\eps(U)+ \frac{\nabla V\otimes \nabla V}2\right)\sd x' + \frac1{24} \int_{\Omega'} Q_2(\nabla^2 V)\sd x',
\end{equation*}
where $\eps(U)= \sym(\nabla U)$,
\begin{eqnarray}
  \label{eq:1}
  U&=& \lim_{h\to 0} \frac1{h^\gamma}
  \left(
\begin{pmatrix}
    y_1^h\\y_2^h
  \end{pmatrix}
  - \operatorname{Id}'
\right),\quad
  V= \lim_{h\to 0} \frac1{h^\delta} y_3^h,\\\label{eq:2}
\delta&=&\alpha-2,\quad \gamma=
\begin{cases}
2(\alpha-2) &\text{if}\ 2<\alpha\leq 3\\
  \alpha-1 &\text{if}\ \alpha> 3
\end{cases}
, 
\end{eqnarray}
where $\operatorname{Id}'(x)= (x_1,x_2)^T$
and $Q_2\colon \R^{2\times 2}\to \R$ is related to $Q_3(F):= D^2W(\Id)(F,F)$ by
\begin{equation*}
  Q_2(G)= \min_{a\in\R^3} Q_3(G+a\otimes e_3+ e_3\otimes a).
\end{equation*}
Here
\begin{equation*}
  \Lambda_\alpha =
  \begin{cases}
    +\infty &\text{if}\ 2 <\alpha <3,\\
   1 &\text{if}\ \alpha =3,\\
 0 &\text{if}\ \alpha >3.
  \end{cases}
\end{equation*}
Thus for $2<\alpha<3$ one has the ``geometrically linear'' constraint $2\eps(U)+\nabla V\otimes \nabla V=0$, which again has so far prevented the rigorous study of the associated dynamic wave equation or non-minimizing equilibria.
For $\alpha=3$ (and therefore $\beta=4$) one obtains the von K\'arm\'an plate theory and for $\alpha>3$ (and therefore $\beta>4$) one obtains a linear Euler-Lagrange equation (linear Germain-Lagrange theory), 
which for isotropic materials reduces to the biharmonic equation.

Here we study the cases $\alpha=3, \beta=4$  and $\alpha>3, \beta=2\alpha-2>4$ in the dynamic situation. The equations of elastodynamics arise from the Lagrangian
\begin{equation*}
  \frac1h \int_{\Omega_h} \left(\frac{|z_t|^2}2- W(\nabla z(x))+ f^hz\right)\sd x=  \int_{\Omega} \left(\frac{|y_t|^2}2- W(\nabla_h y(x))+f^h\cdot y\right)\sd x
\end{equation*}
and solutions formally preserve the total energy
\begin{equation}\label{eq:3}
  \int_{\Omega} \left(\frac{|y_t|^2}2+ W(\nabla_h y(x))- f^h\cdot y\right)\sd x, 
\end{equation}
where it is assumed that $f^h$ is independent of time for simplicity.
In view of (\ref{eq:1})-(\ref{eq:2}) we expect that
\begin{alignat*}{3}
  y_3 &\sim h, &\quad
  \begin{pmatrix}
    y_1\\y_2
  \end{pmatrix}
 - \operatorname{Id}'
&\sim h^2 &\quad& \text{for}\ \alpha=3,\beta=4\\
  y_3 &\sim h^{\alpha-2}, &\quad
  \begin{pmatrix}
    y_1\\y_2
  \end{pmatrix}
-\operatorname{Id}'
 &\sim h^{\alpha-1} &\quad& \text{for}\ \alpha>3,\beta=2\alpha-2>4
\end{alignat*}
The idea to balance the kinetic and potential energy in (\ref{eq:3}) suggests to rescale time as $\tau=ht$ if $\alpha=3$. Then the total energy becomes
\begin{equation*}
  E_{\rm tot} = h^4\int_{\Omega} \left(\frac{|\partial_\tau \frac{y}h|^2}2+ \frac1{h^4}W(\nabla_h y(x))- \frac{f^h_3}{h^3}\frac{y_3}h\right)\sd x
\end{equation*}
and with $\tilde{f}_h= h^{-3}f_3^he_3$ the evolution equation is 
\begin{equation*}
 \frac1{h^2}\partial_\tau^2y-\frac1{h^4} \Div_h DW(\nabla_h y)= \frac1h \tilde{f}_h.
\end{equation*}
or equivalently
\begin{equation}\label{eq:4}
 \partial_\tau^2y-\frac1{h^2} \Div_h DW(\nabla_h y)= h \tilde{f}_h,
\end{equation}
where $\tilde{f}_h \sim 1$ as $h\to 0$. Additionally we assume Neumann boundary conditions at $x_d=\pm \frac12$ and periodic boundary conditions in tangential direction.
In the case $\alpha=3$ we will show existence of strong solutions of (\ref{eq:4}) for well-prepared and small data in a natural scaling with respect to $h$ and  time $\tau \in (0,T_0)$. In particular we assume that the rescaled $\tilde{f}_h$ is small, cf. Section~\ref{sec:MainResult} below. -- Note that the small time interval $(0,T_0)$ for $\tau$ turns over to a large time interval $(0,T_0h^{-1})$ in the original time scale for $t$. In the case $\alpha>3$, we will use the same time scale. Then we are able to show existence of strong solutions for $\tau \in (0,T)$ for any $T>0$ provided that $\tilde{f}_h \sim h^{\alpha-3}$ and suitable initial data, cf. Section~\ref{sec:MainResult} below.  In this case we are even able to construct the leading term of the solution $y=y_h$ as $h\to 0$ provided $W(F)=\dist(F,SO(3))^2$, cf. Section~\ref{sec:Asymptotics}. 

Together with \cite{AbelsMoreMuellerGammaConv} this shows that after the natural time rescaling and for well prepared data of the correct size solutions of the $3$-d nonlinear elastodynamics converge to solutions of the dynamic von K\'arm\'an equation or linear von K\'arm\'an equation depending on the size of the data. 
We note that a similar result in the case of stationary solutions was shown  by Monneau~\cite{MonneauJustification} if the limit system are the von K\'arm\'an plate equations. 
Ge, Kruse and Marsden~\cite{GeKruseMarsden} have taken an alternative and very general
approach to study the limit from three-dimensional
elasticity to shells and rods by establishing convergence of the
underlying Hamiltionian structure.
This suggests, but does not prove the convergence of the corresponding
dynamical problems
(see e.g. recent work by Mielke~\cite{MielkeWeakConvMethods} for the question on the relation of the
convergence of the Hamiltonian
and the convergence of the resulting dynamical problems).
General information and many further references
on the dynamics of lower-dimensional nonlinear elastic structures can be
found in the book by Antman~\cite{Antman}.
For results on existence of weak and strong solutions of the non-stationary von K\'arm\'an plate equations we refer to e.g. Chen and Wahl~\cite{ChenWahl}, Koch and Lasiecka~\cite{KochLasiecka}, Lasiecka~\cite{LasieckaStabilization}, Koch and Stahel~\cite{KochStahel}. For a survey on results and open problem of nonlinear elasticity, stationary and non-stationary, we refer to Ball~\cite{BallOpenProblems}.

Let us explain the strategy of our proof and the main difficulties. Basically, the strong solutions are constructed by the energy method as presented in Koch~\cite{KochWaveEquation}  for the case of Neumann boundary conditions. (See the book by Majda~\cite{MajdaCompFluidFlow} for the full space case or  the classical paper by Hughes et al.~\cite{MarsdenEtAl} for a more abstract and general version. See Kikuchi and Shibata ~\cite{ShibataKikuchi} for a different approach.) Essentially existence of strong solutions for fixed $h>0$ and some $T>0$ depending on $h$ follows from \cite{KochWaveEquation}. Although the latter results are proved for the case of a smooth bounded domain, the proofs easily carry over to the present situation (for every fixed $h>0$) and many arguments even simplify in our situation since the boundary is flat and homogeneous Neumann boundary conditions are considered. Hence the main novelty of this contribution is the proof that for appropriately scaled initial data the maximal time of existence is bounded below by a positive constant as $h\to 0$. 

To explain the main new difficulties in the following let us recall the energy method briefly. The starting point in the method is the conservation of energy:
\begin{equation*}
  \frac{d}{dt} \left(\frac12\|\partial_t y(t)\|_{L^2(\Omega)}^2+\frac1{h^2}\int_\Omega W(\nabla_h y)\sd x\right) - \int_\Omega h\tilde{f}_h\cdot \partial_t y(t)\sd x = 0
\end{equation*}
which follows from (\ref{eq:4}) by multiplication with $\partial_t y$ under appropriate boundary conditions. (Here and in the following we replace $\tau$ by $t$.) Moreover, differentiating (\ref{eq:4}) with respect to $x$ one gets a control of 
\begin{equation}\label{eq:5}
  \frac{d}{dt} \left(\frac12\|\partial_t \partial_x^\beta y(t)\|_{L^2(\Omega)}^2+\frac1{h^2}\int_\Omega D^2W(\nabla_h y)\partial_x^\beta \nabla_h y:\partial_x^\beta\nabla_h y \sd x\right) = R_\beta,
\end{equation}
where the remainder term $R_\beta$ can be controlled with the aid of the Gronwall inequality once the left hand side controls $\partial_x^\beta \nabla_h y$ suitably. To this end it is essential to have the coercive estimate  
\begin{equation}\label{eq:coercive}
  \frac1{h^2}\int_\Omega D^2W(\nabla_h y)\nabla_h w:\nabla_h w\sd x\geq c_0 \left\|\frac1h \eps_h(w)\right\|_{L^2(\Omega)}^2 
\end{equation}
where $\eps_h(w) = \sym (\nabla_h w)$, cf. (\ref{eq:UniformCoercive'}) below. By Korn's inequality in the present $h$-dependent version we have
\begin{equation*}
  \|\nabla_h w\|_{L^2(\Omega)} \leq C\left\|\frac1h\eps_h (w)\right\|_{L^2(\Omega)},
\end{equation*}
cf. Lemma~\ref{lem:Korn} below. Therefore we will have one order of $h$ better decay of the symmetric part of $\nabla_h y$ than for the full gradient/the skew-symmetric part. To obtain (\ref{eq:coercive}) (and similar estimates) it will be essential that 
$$\frac1h\|\eps_h (y)-I\|_{L^\infty}+ \|\nabla_h y-I\|_{L^\infty}\leq \eps h$$ 
for some sufficiently small $\eps>0$ and to treat the symmetric and asymmetric part carefully in a Taylor expansion of $D^2W(\nabla_h y)$ around $I$, cf. Sections~\ref{sec:Preliminaries} and \ref{sec:Linear} for the details.

Several technical difficulties arise from the fact that we are dealing with natural boundary conditions at the upper and lower boundary $x_d=\pm \frac12$. In tangential direction we assume periodic boundary conditions. First of all, in this situation it is easy to differentiate in tangential and temporal direction to obtain (\ref{eq:5}) with $\partial_x^\beta w$ replaced by $\partial_z^\beta w$, where $z=(x',t)$ and $x'=(x_1,\ldots ,x_{d-1})$. Therefore we are using anisotropic $L^2$-Sobolev spaces of sufficiently high order to control $\nabla_h y$ in $L^\infty$. In particular, one of the basic spaces is 
\begin{equation*}
  \tilde{V}(\Omega)=\left\{u\in L^2(\Omega): \nabla u, \partial_{x_j}\nabla u\in L^2(\Omega), j=1,\ldots, d-1\right\}\hookrightarrow L^\infty(\Omega)
\end{equation*}
if $d=2,3$. Note that $\tilde{V}(\Omega)$ is slightly larger than $H^2(\Omega)$ and that $u\in H^2(\Omega)$ if and only if $u\in \tilde{V}(\Omega)$ and $\partial_{x_d}^2 u\in L^2(\Omega)$. Moreover, since we are dealing with natural boundary conditions, we want to keep the equation in divergence form. Therefore we do not use the identity
\begin{equation*}
  \Div_h DW(\nabla_hy)= D^2W(\nabla_hy)\cdot \nabla_h^2y
\end{equation*}
to obtain a quasi-linear system. Instead we differentiate (\ref{eq:4}) with respect to time or tangentially and solve
\begin{equation*}
  \partial_t^2 w_j -\frac1{h^2}\Div_h \left(D^2W(\nabla_hy)\nabla_hw_j\right)=hf_j,\quad j=0,\ldots d-1
\end{equation*}
where $w_0=\partial_t y$, $f_0=\partial_t \tilde{f}_h$, $w_j=\partial_{x_j} y$, $f_j=\partial_{x_j}\tilde{f}_h$ for $j=1,\ldots,d-1$. Applying suitable $h$-uniform estimates for the linearized system, we prove that the solutions cannot blow up on a time interval independent of $0<h\leq 1$ if the data are sufficiently small.

The structure of the article is as follows: In Section~\ref{sec:Preliminaries} we introduce some notation and derive some preliminary results. Our main result is presented in Section~\ref{sec:MainResult}.  The essential results for the linearized system are derived in Section~\ref{sec:Linear}. These are applied in Section~\ref{sec:uniform}, where our main result is proved. Finally, in Section~\ref{sec:Asymptotics} we derive a first order asymptotic expansion as $h\to 0$ in the case that the limit system is linear, i.e., $\beta>4$, and $W(F)= \dist(F,SO(d))^2$.

\medskip

\noindent
{\bf Acknowledgements:} This work was partially supported by GNAMPA, through the project ``Problemi di riduzione di dimensione per strutture elastiche sottili'' 2008. Moreover, we are grateful to an anonymous referee for helpful remarks on  the content of paper and further references, which helped to improve the paper.

\section{Notation and Preliminaries}\label{sec:Preliminaries}

For any measurable set $M\subseteq \R^N$ the inner product of $L^2(M)$ (w.r.t. to Lebesgue measure) is denoted by $(.,.)_M$. Moreover, $H^k(\Omega)$, $k\in\N_0$, denotes the usual $L^2$-Sobolev spaces. If $X$ is a Banach space, then the vector-valued variants of $L^2(M)$ and $H^k(M)$ are denoted by $L^2(M;X)$, $H^k(M;X)$, respectively. Furthermore, $C^k([0,T];X)$, $k\in \N_0$, denotes the space of all $k$-times continuously differentiable functions $f\colon [0,T]\to X$. 

For the following $\Omega= (-L,L)^{d-1}\times(-\frac12,\frac12)$, $\Omega'=(-L,L)^{d-1}$, $d=2,3$, $x=(x',x_d)$,
where $x'\in \R^{d-1}$, let $\nabla_h = (\nabla_{x'},
\frac1h \partial_{x_d})^T$, $\nabla_{x,t}=(\partial_t, \nabla_x)$ and let
\begin{equation*}
  \eps_h (w) = \sym(\nabla_h w),\qquad \eps(w)=\eps_1(w),
\end{equation*}
if $w\colon M\subset \R^d\to \R^d$ is a suitable vector field. Here $\sym A=
\frac12 (A+A^T)$ and we denote $\skw A:= \frac12 (A-A^T)$. Moreover, we denote $z=(t,x')$, where $z_0 = t$ and $z_j = x_j$ for $j=1,\ldots, d-1$.

For $s> 0$, $s\not\in \N_0$, we define $L^2$-Bessel potential spaces 
\begin{equation*}
  H^s(\Omega) = \{f\in L^2(\Omega): f= F|_{\Omega}\ \text{for some}\ F\in H^s(\R^d)\}
\end{equation*}
as usual by restriction, equipped with the quotient norm.
Since $\Omega$ is a Lipschitz domain, there is a continuous extension operator $E$ such that $E\colon H^k(\Omega)\to H^k(\R^d)$ for all $k\in \N$, cf. Stein~\cite[Chapter VI, Section 3.2]{Stein:SingInt}. Hence $H^s(\Omega)$, $s\geq 0$, is retract of $H^s(\R^d)$ and we obtain the usual interpolation properties, cf. e.g. \cite{Triebel1}. In particular, we have
\begin{equation}\label{eq:InterpolHs}
  (H^{s_0}(\Omega), H^{s_1}(\Omega))_{\theta,2}= H^s(\Omega),\qquad s= (1-\theta)s_0+ \theta s_1, 
\end{equation}
for all $\theta \in (0,1)$, $s\geq 0$, where  $(.,.)_{\theta,p}$ denotes the real interpolation method.

If $0<T\leq \infty$ and $X$ is a Banach space, then  $BUC([0,T];X)$ is the space of all bounded and uniformly continuous functions $f\colon [0,T)\to X$.
Now let $X_0,X_1$ be Banach spaces such that $X_1\hookrightarrow X_0$ densely. 
Then
\begin{equation}
  \label{eq:BUCEmbedding}
   W^1_p(0,T;X_0) \cap L^p(0,T;X_1) \hookrightarrow BUC([0,T];(X_0,X_1)_{1-\frac1p,p})
\end{equation}
for all $1\leq p <\infty$
continuously, cf. Amann~\cite[Chapter III, Theorem 4.10.2]{Amann}.
If $X_0=H$ is a Hilbert space and $H$ is identified with its dual, then $X_1\hookrightarrow H\hookrightarrow X_1'$ and
\begin{equation}
  \label{eq:HNormDifferential}
  \frac12\frac{d}{dt} \|f\|_{H}^2 = \weight{\frac{d}{dt} f(t), f(t)}_{X_1',X_1} \qquad \text{for almost all}\ t\in [0,T] 
\end{equation}
provided that $f\in L^p(0,T;X_1)$ and $\frac{d}{dt} f\in L^{p'}(0,T;X_1')$, $1< p <\infty$, cf. Zeidler~\cite[Proposition 23.23]{ZeidlerIIa}. 
In particular, (\ref{eq:HNormDifferential}) implies 
\begin{equation}\label{eq:StrongBUCConv}
  \sup_{t\in [0,T]}\|f(t)\|_{H}^2
 \leq 2\left(\|\partial_t f\|_{L^2(0,T;X_1')}\|f\|_{L^2(0,T;X_1)} + \|f(0)\|_H^2\right).
\end{equation}
Replacing $f(t)$ by $tf(t)$ and $(T-t)f(T-t)$, one easily derives from the latter estimate
\begin{equation}\label{eq:StrongBUCConv'}
  \sup_{t\in [0,T]}\|f(t)\|_{H}
 \leq C_T\|f\|_{H^1(0,T;X_1')}^{\frac12}\|f\|_{L^2(0,T;X_1)}^{\frac12} 
\end{equation}
for some $C_T>0$ depending on $T>0$.

In the following $\mathcal{L}^n(V)$, $n\in\N$, denotes the space of all
$n$-linear mappings $A\colon V^n\to \R$ for a vector space $V$. Moreover, if $A\in
\mathcal{L}^n(V)$, $n\geq 2$, and $x_1,\ldots, x_k\in V$, $1\leq k\leq n$, then $A[x_1,\ldots, x_k] \in
\mathcal{L}^{n-k}(V)$ is defined by $A[x_1,\ldots, x_k](x_{k+1},\ldots, x_n)= A(x_1,\ldots,
x_n)$ for all $x_{k+1},\ldots,x_n\in V$. 

We introduce the scaled inner product 
\begin{equation*}
  A:_h B = \frac1{h^2} \sym A : \sym B + \skw A:\skw B, \quad A,B\in \R^{d\times d}, 0<h\leq 1,
\end{equation*}
and $|A|_h= \sqrt{A:_hA}$ where $A:B= \sum_{i,j=1}^d a_{ij} b_{ij}$. This choice of inner product is motivated by the Korn inequality in thin domains, see Lemma~\ref{lem:Korn} below. Of
course, $:_1$ coincides with the usual inner product $:$ on $\R^{d\times d}$
and therefore $|A|_1=|A|$. For $W\in
\mathcal{L}^n(\R^{d\times d}) $ we define the induced scaled norm by
\begin{equation*}
  |W|_{h} = \sup_{|A_j|_h\leq 1, j=1,\ldots, n} |W(A_1,\ldots,A_n)|.
\end{equation*}
Note that, since $|A|_h\geq |A|_1=|A|$ for all $A\in \R^{d\times d}$, we have $|W|_h\leq |W|_1=:|W|$ for any $W\in
\mathcal{L}^n(\R^{d\times d}) $ and $0< h\leq 1$.

As usual we
identify $\mathcal{L}^1(\R^{d\times d})= (\R^{d\times d})'$ with $\R^{d\times
  d}$. But one has to be careful whether this representation is taken with
respect to the usual scalar product $:$ on $\R^{d\times d}$ or with respect to
$:_h$, i.e., $W\in \mathcal{L}^1(\R^{d\times d})$ is identified with $A\in
\R^{d\times d}$ such that 
\begin{equation*}
  W(B) = A:_hB\qquad \text{for all}\ B\in \R^{d\times d}.
\end{equation*}
If nothing else is mentioned, we identify $(\R^{d\times d})'$ and $\R^{d\times
  d}$ using the standard inner product $:$. In particular, if $W\in C^1(U)$, $U\subset \R^{d\times d}$ and $A\in U$, then
$DW(A)\in (\R^{d\times d})'\cong \R^{d\times d}$ coincides with
\begin{equation*}
  DW(A): B= \left. \frac{d}{dt}W(A+tB)\right|_{t=0}\qquad \text{for all}\
  B\in\R^{d\times d}. 
\end{equation*}
Furthermore, $W\in \mathcal{L}^2(\R^{d\times d})$ is usually identified with the linear mapping $\tilde{W}\colon \R^{d\times d}\to \R^{d\times d}$ defined by
\begin{equation*}
  \tilde{W}A: B = W(A,B)\qquad \text{for all}\ A,B\in \R^{d\times d}.
\end{equation*}
Finally, we denote by
\begin{equation*}
  \|W\|_{L^p_h(M;\mathcal{L}^n(\R^{d\times d}))}\equiv \|W\|_{L^p_h(M)} =
  \left(\int_M |W(x)|_h^p\sd x\right)^{\frac1p}
\end{equation*}
if $1\leq p<\infty$ and with the obvious modifications if $p=\infty$. Here
$M\subseteq \R^d$ is measurable. Moreover, for $f\in L^p(M;\R^{d\times d})$ the scaled norm
$\|f\|_{L^p_h(M;\R^{d\times d})}\equiv \|f\|_{L^p_h(M)}$ is defined in the same way.
 
We now state the relevant  Korn inequality in thin domains.
\begin{lemma}\label{lem:Korn}
  There is a constant $C$ such that
  \begin{equation}\label{eq:Korn}
    \|\nabla_h u \|_{L^2(\Omega)} \leq C \left\|\frac1h\eps_h(
    u)\right\|_{L^2(\Omega)}
  \end{equation}
  for all $0<h\leq 1$ and $u\in H^1(\Omega)^d$ such that $u|_{x_j=-L}=u|_{x_j=L}$,
  $j=1,\ldots, d-1$.
\end{lemma}
\begin{proof}
For clamped boundary conditions the Korn inequality in thin domains was
 proved by Kohn and Vogelius~\cite[Prop. 4.1]{KohnVogeliusII}. They mention that
 the result also holds without boundary conditions, modulo infinitesimal
 rigid motions.
 For the convenience of the reader we provide a proof of Lemma~\ref{lem:Korn}.

First we prove the case $d=2$. Let $\Omega_h:=(-L,L)^{d-1}{\times}(-\frac{h}{2},\frac{h}{2})$ and let $u\in
H^1(\Omega_h;\R^2)$ satisfy the boundary conditions  $u|_{x_j=-L}=u|_{x_j=L}$,
  $j=1,\ldots, d-1$.
First of all by a simple scaling in $x_d$, (\ref{eq:Korn}) is equivalent to
  \begin{equation}\label{eq:Korn2}
    \|\nabla u \|_{L^2(\Omega_h)} \leq \frac{C}h \|(\nabla u)_{sym}\|_{L^2(\Omega_h)}
  \end{equation}
Let $N_h$ be the integer part of $\frac{2L}{h}$ and let $\ell_h:=\frac{2L}{N_h}$.
We set $J_h:=\{-L+k\ell_h: k=0,\dots,N_h-1\}$. 
By applying Korn inequality on the set $(a, a+\ell_h){\times}(-\frac{h}{2},\frac{h}{2})$
for every $a\in J_h$, we can construct a piecewise constant function $A\colon(-L,L)\to \mtwo$ such that $A(x_d)$ is skew-symmetric and
\begin{equation}\label{pwkorn}
\int_{\Omega_h}|\nabla u- A|^2\, dx \leq C \int_{\Omega_h}|\eps(u)|^2\, dx. 
\end{equation}
Note that, since $\frac{\ell_h}h$ is bounded from above and from below, we can use the same Korn inequality constant on each set 
$(a, a+\ell_h){\times}(-\frac{h}{2},\frac{h}{2})$.

We claim that
\begin{equation}\label{Aest}
\int_{\Omega_h}|A(x_1)-A_0|^2\, dx\leq \frac{C}{h^2}\int_{\Omega_h}|\eps(u)|^2\, dx. 
\end{equation}
where $A_0:=A(-L)$.

Let us fix $a\in J_h$ and let $b:=a+\lambda\ell_h$, with $\lambda\in\{0,1\}$.
By applying Korn inequality on the set $(a, a+2\ell_h){\times}(-\frac{h}{2},\frac{h}{2})$ 
we have that there exists $\tilde A\in\mtwo$ such that
$$
\int_{(a, a+2\ell_h){\times}(-\frac{h}{2},\frac{h}{2})}|\nabla u- \tilde A|\, dx \leq C 
\int_{(a, a+2\ell_h){\times}(-\frac{h}{2},\frac{h}{2})}|\eps(u)|^2\, dx.
$$
From this inequality we deduce
\begin{eqnarray*}
h\ell_h|A(b)-\tilde A|^2
& \leq &
2  \int_{(b, b+\ell_h){\times}(-\frac{h}{2},\frac{h}{2})}|\nabla u- A(x_1)|^2\, dx
\\
& & + 2 \int_{(b, b+\ell_h){\times}(-\frac{h}{2},\frac{h}{2})}|\nabla u-\tilde A|^2\, dx
\\
& \leq & C \int_{(a, a+2\ell_h){\times}(-\frac{h}{2},\frac{h}{2})}|\eps(u)|^2\, dx.
\end{eqnarray*}
Combining the previous inequality for $\lambda=0$ and $\lambda=1$, we obtain
\begin{eqnarray*}
h\ell_h|A(a)-A(b)|^2
& \leq & 2h\ell_h( |A(a)-\tilde A|^2 + |A(b)-\tilde A|^2)
\\
& \leq & C \int_{(a, a+2\ell_h){\times}(-\frac{h}{2},\frac{h}{2})}|\eps(u)|^2\, dx.
\end{eqnarray*}
As $A$ is constant on each interval $(a, a+\ell_h)$, this is equivalent to say that
\begin{equation}\label{diffquot}
\int_{(a, a+\ell_h){\times}(-\frac{h}{2},\frac{h}{2})}|A(x_1+\ell_h)-A(x_1)|^2\, dx
\leq C \int_{(a, a+2\ell_h){\times}(-\frac{h}{2},\frac{h}{2})}|\eps(u)|^2\, dx.
\end{equation}

Let us set $I_{k,j}:=-L+\ell_h(k,k+j)$.
By convexity we have the following estimate: 
\begin{eqnarray*}
\lefteqn{\int_{\Omega_h}|A(x_1)-A_0|^2\, dx =
h\sum_{k=0}^{N_h-1}\int_{I_{k,1}}|A(x_1)-A_0|^2\, dx_1}
\\
& = & h\sum_{k=0}^{N_h-1}\int_{I_{k,1}}
\Big|\sum_{m=0}^{k-1}\big( A(x_1-m\ell_h)-A(x_1-(m+1)\ell_h)\big)\Big|^2\, dx_1
\\
& \leq & h\sum_{k=0}^{N_h-1} k \sum_{m=0}^{k-1} \int_{I_{k,1}}
\Big|A(x_1-m\ell_h)-A(x_1-(m+1)\ell_h)\Big|^2\, dx_1.
\end{eqnarray*}
By \eqref{diffquot} we deduce 
$$
\int_{\Omega_h}|A(x_1)-A_0|^2\, dx \leq \sum_{k=0}^{N_h-1} k \sum_{m=0}^{k-1} 
C\int_{I_{k-m-1, 2}{\times}(-\frac{h}{2},\frac{h}{2})}
|\eps(u)|^2\, dx.
$$
It is easy to see that for every $k=0,\dots, N_h-1$
$$
\sum_{m=0}^{k-1} 
\int_{I_{k-m-1, 2}{\times}(-\frac{h}{2},\frac{h}{2})}
|\eps(u)|^2\, dx
\leq 2 \int_{\Omega_h} |\eps(u)|^2\, dx.
$$
Therefore, we conclude that
$$
\int_{\Omega_h}|A(x_1)-A_0|^2\, dx \leq C N_h^2\int_{\Omega_h} |\eps(u)|^2\, dx,
$$
which proves claim \eqref{Aest}.

Combining \eqref{pwkorn} and \eqref{Aest}, we conclude that for every $u\in H^1(\Omega_h;\R^2)$
there exists a constant skew-symmetric $A_0\in\mtwo$ such that
$$
\int_{\Omega_h}|\nabla u-A_0|^2\, dx\leq \frac{C}{h^2}\int_{\Omega_h}|\eps(u)|^2\, dx.
$$

Since 
$$
\int_{\Omega_h}\Big|\frac1{|\Omega_h|}\fint_{\Omega_h}({\rm skw}\nabla u)\, dx -A_0 \Big|^2\, dx \leq
\int_{\Omega_h}|({\rm skw}\nabla u)-A_0|^2\,dx ,
$$ 
we also have that
\begin{equation}\label{korn2}
\int_{\Omega_h}\Big|\nabla u-\frac1{|\Omega_h|}\fint_{\Omega_h}({\rm skw}\nabla u)\Big|^2\, dx
\leq \frac{C}{h^2}\int_{\Omega_h}|\eps(u)|^2\, dx 
\end{equation}
for every $u\in H^1(\Omega_h;\R^2)$.

Now, if $u$ is periodic in tangential direction, then
\begin{eqnarray*}
\int_{\Omega_h}\Big|\frac1{|\Omega_h|}\fint_{\Omega_h}({\rm skw}\nabla u)\Big|^2\,dx 
& = &
\int_{\Omega_h}\Big|\frac1{|\Omega_h|}\fint_{\Omega_h}\partial_2 u_1\Big|^2\, dx
\\
& = &\int_{\Omega_h}\Big|\frac1{|\Omega_h|}\fint_{\Omega_h}(\partial_2 u_1+\partial_1 u_2)\Big|^2\, dx
\\
& \leq & \int_{\Omega_h}|\eps(u)|^2\, dx,
\end{eqnarray*}
which, together with \eqref{korn2}, provides us with the desired inequality.

In order to prove the case $d=3$, we use that (\ref{eq:Korn}) for $d=2$
implies
\begin{equation*}
  \left\|
    \begin{pmatrix}
      \partial_{x_j}\\ \frac1{h} \partial_{x_3}
    \end{pmatrix}
    \begin{pmatrix}
      u_j\\ u_3
    \end{pmatrix}
\right\|_{L^2(\Omega)} \leq \frac{C}h   \left\|
    \left(\begin{pmatrix}
      \partial_{x_j}\\ \frac1{h} \partial_{x_3}
    \end{pmatrix}
    \begin{pmatrix}
      u_j\\u_3
    \end{pmatrix}
\right)_{sym} 
\right\|_{L^2(\Omega)} \leq \frac{C}h\|(\nabla_h u)_{sym}\|_{L^2(\Omega)} 
\end{equation*}
for $j=1,2$ and any $u\in H^1(\Omega)^3$. Moreover, applying Korn's inequality in
$(-L,L)^2$ with periodic boundary conditions, we obtain
\begin{equation*}
  \|\nabla_{x'} u'\|_{L^2(\Omega)} \leq C \|(\nabla_{x'}
  u')_{sym}\|_{L^2(\Omega)}
  \leq C\|(\nabla_x u)_{sym}\|_{L^2(\Omega)},
\end{equation*}
where $u'= (u_1,u_2)^T$. Altogether this proves (\ref{eq:Korn}) for $d=3$.
\end{proof}
\begin{remark}
  The latter lemma shows that $\left\|\frac1h\eps_h(u)\right\|_{L^2(\Omega)}$ is
    equivalent to $\|\nabla_h u\|_{L^2_h(\Omega)}$ with constants independent
    of $0<h\leq 1$.
\end{remark}
 We denote
\begin{equation*}
  H^m_{per}(\Omega)= \left\{f \in H^m(\Omega): \partial_x^\alpha f|_{x_j=-L}= \partial_x^\alpha f|_{x_j=L}, j=1,\ldots,d-1, |\alpha|\leq m-1\right\}.
\end{equation*}

 Throughout this contribution the following anisotropic variant of $H^m_{per}(\Omega)$ will be important:
   \begin{eqnarray*}
     H^{m_1,m_2}(\Omega) &=& \left\{u\in L^2(\Omega): \nabla_{x'}^k\partial_{x_d}^l u\in
       L^2(\Omega),k=0,\ldots,m_1, l=0,\ldots, m_2\right.,\\
      && \left. \ \partial_{x'}^\alpha \partial_{x_d}^lu|_{x_j=-L}=\partial_{x'}^\alpha \partial_{x_d}^lu|_{x_j=L}, j \leq d-1, |\alpha|\leq m_1-1, l\leq m_2\right\}
   \end{eqnarray*}
   where $m_1\in \N,m_2\in\N_0$. The spaces are equipped with the inner product
   \begin{equation*}
     (f,g)_{H^{m_1,m_2}}= \sum_{|\alpha|\leq m_1,k=0,\ldots, m_2} (\partial_{x'}^\alpha \partial_{x_d}^k f,\partial_{x'}^\alpha \partial_{x_d}^k g)_{L^2(\Omega)}
   \end{equation*}
   Please note that periodic boundary conditions are included in the spaces $H^{m_1,m_2}(\Omega)$ in contrast to the space $H^m(\Omega)$, where we denote them by a subscript ``$per$'' in order to be consistent with the usual definition of $H^m(\Omega)$. Moreover, note that $f\in H^{m_1,m_2}(\Omega)$ if and only if its periodic extension $\tilde{f}$ (w.r.t. $x_j$, $j=1,\ldots,d-1$) satisfies
   \begin{equation*}
     \nabla_{x'}^\alpha \partial_{x_d}^l \tilde{f}\in L^2_{loc}(\R^{d-1}\times(-\tfrac12,\tfrac12))\quad \text{for all}\ |\alpha|\leq m_1,l=0,\ldots, m_2.
   \end{equation*}
Therefore we can also identify $f\in H^{m_1,m_2}(\Omega)$ with a function $f\colon \R^{d-1}\times (-\frac12,\frac12)$ that is $2L$-periodic in $x_j$, $j=1,\ldots,d-1$ and satisfies the latter smoothness condition.

Similarly, an anisotropic variant of $L^p$ will be useful:
\begin{eqnarray*}
  L^{p,q}(\Omega)&=& \left\{u\colon \Omega\to \R: \|u(x_1,.)\|_{L^q(-\frac12,\frac12)}\in L^p((-L,L)^{d-1})\right\}
\end{eqnarray*}
where $1\leq p,q\leq \infty$
equipped with the norm
\begin{equation*}
  \|u\|_{L^{p,q}} = \left\|\|u(x_1,.)\|_{L^q(-\frac12,\frac12)}\right\|_{L^p((-L,L)^{d-1})}.
\end{equation*}
We note that from the usual H\"older inequality it follows that
\begin{equation*}
  \|fg\|_{L^{p,q}(\Omega)}\leq \|f\|_{L^{p_1,q_1}(\Omega)}\|g\|_{L^{p_2,q_2}(\Omega)},
\end{equation*}
for all $1\leq p_1,q_1,p_2,q_2\leq \infty$ such that
\begin{equation*}
  \frac1p = \frac1{p_1} +\frac1{p_2},\qquad   \frac1q = \frac1{q_1} +\frac1{q_2}.
\end{equation*}
\begin{lemma}\label{lem:Embeddings}
  Let $d=2,3$. Then
  \begin{equation*}
    H^{1,0}(\Omega) \hookrightarrow L^{p,2}(\Omega), \quad
    H^{2,0}(\Omega) \hookrightarrow L^{\infty,2}(\Omega), \quad
    H^1(\Omega) \hookrightarrow L^{4,\infty}(\Omega) \quad
  \end{equation*}
continuously for $p=\infty$ if $d=2$ and any $1\leq p<\infty$ if $d=3$. 
Finally, let
\begin{equation*}
  V(\Omega):= H^{1,1}(\Omega)\cap H^{2,0}(\Omega).
\end{equation*}
Then $V(\Omega)\hookrightarrow C^0(\ol{\Omega})$ continuously.
\end{lemma}
\begin{proof}
  The first embedding follows from $H^1(\Omega')\hookrightarrow L^p(\Omega')$ and the  second from $H^2(\Omega')\hookrightarrow L^\infty(\Omega')$ since $d=2,3$ and $\Omega'= (-L,L)^{d-1}$. The third embedding follows from
\begin{equation*}
  H^1(-\tfrac12,\tfrac12;L^2(\Omega'))\cap L^2(-\tfrac12, \tfrac12; H^1(\Omega'))\hookrightarrow BUC([-\tfrac12, \tfrac12];H^{\frac12}(\Omega'))
\end{equation*}
and $H^{\frac12}(\Omega')\hookrightarrow L^4(\Omega')$. Finally, the last embedding follows from
  \begin{eqnarray*}
    \lefteqn{L^2(-\tfrac12,\tfrac12; H^{1+k}((-L,L)^{d-1}))\cap
  H^1(-\tfrac12,\tfrac12;H^1((-L,L)^{d-1}))}\\
&& \hookrightarrow BUC([-\tfrac12,\tfrac12];
  H^{1+\frac{k}2}((-L,L)^{d-1}))\hookrightarrow C^0(\ol{\Omega})
  \end{eqnarray*}
where $k=d-2$
  because of (\ref{eq:BUCEmbedding}) and Sobolev
  embeddings.
\end{proof}
\begin{remark}
  The spaces $H^{1,0}(\Omega)$ and $V(\Omega)$ are two fundamental spaces, which will be used to solve the evolution equation. We note that
  \begin{equation*}
    f\in V(\Omega)\quad \Leftrightarrow \quad f, \nabla f \in H^{1,0}(\Omega).
  \end{equation*}
Most of the time we will estimate $f\in V(\Omega)$ by the $h$-dependent norm
\begin{equation*}
\|f\|_{V_h}:=  \|(f,\nabla_h f)\|_{H^{1,0}(\Omega)}.
\end{equation*}
\end{remark}
 Because of the embedding $V(\Omega)\hookrightarrow L^\infty(\Omega)$, we are able to show that $V(\Omega)$ is an algebra with respect to point-wise multiplication. More precisely, we
   obtain:
   \begin{corollary}\label{cor:Algebra} 
Let $d=2,3$. Then there is some $C=C(\Omega)>0$ such that
     \begin{eqnarray}\label{eq:ProdEstim'}
       \|(u_1\cdot v, \nabla_h (u_1\cdot v)\|_{L^2}&\leq& C\|(u_1, \nabla_h u_1)\|_{H^{1,0}}\|(v, \nabla_h v)\|_{L^2}\\\label{eq:ProdEstim}
       \|(u_1\cdot u_2, \nabla_h (u_1\cdot u_2)\|_{H^{1,0}}&\leq& C\|(u_1, \nabla_h u_1)\|_{H^{1,0}}\|(u_2, \nabla_h u_2)\|_{H^{1,0}}
     \end{eqnarray}
for all $u_1,u_2\in V(\Omega)$, $v\in H^1_{per}(\Omega)$ uniformly in $0<h\leq 1$. Moreover,
     if $F\in C^{2}(\ol{U})$ for some open $U\subset \R^N$,
     $N\in\N$, and $u\in V(\Omega)^N$, then for every $R>0$ there is some $C(R)$ independent of $u$ such that
     \begin{equation}\label{eq:CompEstim}
       \|(F(u), \nabla_h F(u))\|_{H^{1,0}(\Omega)} \leq C(R)\qquad \text{if}\ \|(u,\nabla_h u)\|_{H^{1,0}(\Omega)}\leq R
     \end{equation}
     uniformly in $0<h\leq 1$ and if $u(x)\in \ol{U}$ for all $x\in\overline{\Omega}$.
   \end{corollary}
   \begin{proof}
     First of all (\ref{eq:ProdEstim'}) can be derived in a straight forward manner using Lemma~\ref{lem:Embeddings}. Moreover, (\ref{eq:ProdEstim})  follows from (\ref{eq:CompEstim}) by
     first considering $\|u_1\|_{V_h}, \|u_2\|_{V_h}\leq 1$ and
     $F(u_1,u_2) = u_1\cdot u_2$ together with a scaling argument.

     Hence it only remains to prove (\ref{eq:CompEstim}).      
     First of all, 
     \begin{eqnarray*}
       \partial_{x_j} F(u) &=& DF(u) \partial_{x_j} u \\
       \partial_{x_j}\partial_{x_k} F(u) &=& DF(u) \partial_{x_j}\partial_{x_k}
       u + D^2F(u)(\partial_{x_j}u,\partial_{x_k}u)
     \end{eqnarray*}
     where $DF(u),D^2F(u)$ are uniformly bounded since $u\in C^0(\ol{\Omega})$
     and $u(x) \in \ol{U}$ for all $x\in \ol\Omega$. Therefore
     $\nabla_h F(u) \in L^2(\Omega)$ can be easily estimated. Hence it only remains to consider the second order derivatives. 
      To this end we use that
     \begin{eqnarray*}
       \lefteqn{\left\|D^2F(u)(\partial_{x_j}u,\frac1h\partial_{x_d}u)\right\|_{L^2(\Omega)}\leq C
       \|\partial_{x_j}u\|_{L^{4,\infty}(\Omega)}
       \left\|\frac1h\partial_{x_d}u\right\|_{L^{4,2}(\Omega)}}\\
       &\leq & C\|\partial_{x_j} u\|_{H^1(\Omega)}\left\|\frac1h\partial_{x_d}u\right\|_{H^{1,0}(\Omega)}
\leq C'(R)\|(u, \nabla_h u)\|_{H^{1,0}(\Omega)}
     \end{eqnarray*}
     for all $j=1,\ldots,d-1$ due to Lemma~\ref{lem:Embeddings}. Similarly,
     \begin{eqnarray*}
       \left\|D^2F(u)(\partial_{x_j}u,\partial_{x_k}u)\right\|_{L^2(\Omega)} &\leq& C'(R)\|(u, \nabla_h u)\|_{H^{1,0}(\Omega)}
     \end{eqnarray*}
for all $j,k=1,\ldots, d-1$. From these estimates the statement of the corollary easily follows.
   \end{proof}

For the following let 
$W\colon B_r(I)\subset\R^{d\times d}\to \R$ be a smooth function for some $r>0$ which is frame invariant, i.e., $W(RF)=W(F)$ for every $F\in \R^{d\times d}$ and $R\in SO(d)$, and such that $DW(I)=0$ and $D^2W(I)\colon \R^{d\times d}\to\R^{d\times d}$ is positive definite on symmetric matrices. Moreover, we set
$\widetilde{W}(G)= W(I+G)$.
The estimates of derivatives of $D^2 \widetilde{W}(\nabla_h u)$ will be essential for the proof of our main result and will be based on the following lemma:
\begin{lemma}\label{lem:D3G}
  There is some constant $C>0$, $\eps>0$, and $A\in C^\infty(\ol{B_\eps(0)};\mathcal{L}^3(\R^{d\times d}))$ such that for all $G\in
  \R^{d\times d}$ with $|G|\leq \eps$ we have
  \begin{equation*}
    D^3\widetilde{W}(G) = D^3\widetilde{W}(0) + A(G),
  \end{equation*}
  where
  \begin{alignat*}{2}
    |D^3\widetilde{W}(0)|_h&\leq Ch &\quad& \text{for all}\ 0<h\leq 1,\\
    |A(G)| &\leq C|G| &\quad& \text{for all}\ |G|\leq \eps. 
  \end{alignat*}
\end{lemma}
\begin{proof}
  First of all, if $|G|\leq \eps$ for $\eps>0$ sufficiently small, we can use a polar decomposition $I+G= RU$, where $R\in SO(d)$ and $U$ is symmetric and positive definite such that $U^2=
  (I+G)^T(I+G)$. From frame invariance we conclude
  that $W(I+G)=W(U)= \widehat{W}(U^2)=\widehat{W}(I+2\sym G+G^TG)$ for some smooth $\widehat{W}\colon V \subset \R^{d\times d}\to \R$, where $V$ is some open neighborhood of $I$. For this proof we denote $A_s= \sym A$. Straight-forward calculations yield
  \begin{eqnarray*}
    DW(F)(H) &=& D\widehat{W}(U^2)(2 H_s + H^TG+G^TH)\\
    D^2W(F)(H_1,H_2) &=& D^2\widehat{W}(U^2)(2H_{1,s} + H_1^TG+G^TH_1,2
    H_{2,s} + H_2^TG+G^TH_2)\\
    && + D\widehat{W}(U^2)(H_1^TH_2+H_2^TH_1)
\end{eqnarray*}
and
\begin{eqnarray*}
    \lefteqn{D^3W(F)(H_1,H_2,H_3)=}\\
    && D^3\widehat{W}(U^2)(2H_{1,s} + H_1^TG+G^TH_1,2
    H_{2,s} + H_2^TG+G^TH_2,2H_{3,s} + H_3^TG+G^TH_3 )\\
    && + D^2\widehat{W}(U^2)(H_1^TH_2+H_2^TH_1, 2H_{3,s}+ H_3^TG+G^TH_3)\\
&&    + D^2\widehat{W}(U^2)(H_1^TH_3+H_3^TH_1, 2H_{2,s}+ H_2^TG+G^TH_2)\\
&&+ D^2\widehat{W}(U^2)(H_2^TH_3+H_3^TH_2, 2H_{1,s}+ H_1^TG+G^TH_1)
  \end{eqnarray*}
  where $F=I+G$. 
From the latter identities the statements immediately follow.
\end{proof}
For the following we denote
\begin{eqnarray*}
  \|A\|_{H^{m_1,m_2}_h} &:=& \left(\sum_{|\alpha|\leq m_1, j=0,\ldots, m_2} \|\partial_{x'}^\alpha \partial_{x_d}^j A\|_{L^2_h(\Omega)}^2\right)^{\frac12}\\
  \|A\|_{H^{m}_h} &:=& \left(\sum_{|\alpha|\leq m} \|\partial_{x}^\alpha A\|_{L^2_h(\Omega)}^2\right)^{\frac12}
\end{eqnarray*}
where $m,m_1,m_2\in \N_0$ and $A\in H^{m_1,m_2}(\Omega)^{d\times d}$, $A\in H^{m}(\Omega)^{d\times d}$, respectively.
\begin{corollary}\label{cor:D3W} 
  There are some $\eps,C>0$ such that
  \begin{eqnarray}\nonumber
    \lefteqn{\|D^3 \widetilde{W}(Z)(Y_1, Y_2, Y_3)\|_{L^1(\Omega)}}\\\label{eq:D3WEstim}
    &\leq & C h \left(\left\|Y_1 \right\|_{H^{1,1}_h} + \left\|Y_1 \right\|_{H^{2,0}_h} \right)\left\|Y_2\right\|_{L^2_h(\Omega)} \left\|Y_3\right\|_{L^2_h(\Omega)} 
  \end{eqnarray}
  for all $Y_1\in V(\Omega)^{d\times d}, Y_2,Y_3 \in
  L^2(\Omega)^{d\times d} $, $0<h\leq 1$ and $\|Z\|_{L^\infty(\Omega)}\leq \min (\eps,h)$ and 
  \begin{eqnarray}\nonumber
    \lefteqn{\|D^3 \widetilde{W}(Z)(Y_1, Y_2, Y_3)\|_{L^1(\Omega)}}\\\label{eq:D3WEstimb}
    &\leq & C h \left\|Y_1\right\|_{H^1_h(\Omega)}\left\|Y_2\right\|_{H^{1,0}_h(\Omega)} \left\|Y_3\right\|_{L^2_h(\Omega)} 
  \end{eqnarray}
  for all $Y_1\in H^1(\Omega)^{d\times d}, Y_2\in H^{1,0}(\Omega)^{d\times d},Y_3 \in
  L^2(\Omega)^{d\times d} $, $0<h\leq 1$ and $Z\in L^\infty(\Omega)^{d\times d}$ with $\|Z\|_{L^\infty(\Omega)}\leq \min(\eps,h)$.
\end{corollary}
\begin{proof}
  The statement follows directly from Lemma~\ref{lem:D3G}, Korn's
  inequality due to Lemma~\ref{lem:Korn}, and Lemma~\ref{lem:Embeddings}.
\end{proof}

\section{Long-Time Existence for Thin Rods/Plates}
\subsection{Main Result}
\label{sec:MainResult}
We consider
\begin{equation}\label{eq:Evol1}
  \partial_t^2 u_h - \frac1{h^2} \Div_h D\widetilde{W}(\nabla_h u_h) = f_h h^{1+\theta} \qquad \text{in}\ \Omega\times I
\end{equation}
where $\widetilde{W}(G)= W(I+G)$,  $\Omega= (-L,L)^{d-1}\times
(-\frac12,\frac12)$, $\beta=4+2\theta$, which is equivalent to $\theta= \alpha -3$, and $I=[0,T_\ast]$ for some $T_\ast>0$ together with the initial and boundary conditions
\begin{eqnarray}\label{eq:BC1}
  \left. D\widetilde{W}(\nabla_h u_h)e_d\right|_{x_d=\pm \frac12} &=&0, \\\label{eq:BC2}
   u_h &\text{is}& 2L\text{-periodic w.r.t.}\ x_j, j=1,\ldots,d-1,\\\label{eq:BC3}
  \left. (u_h,\partial_t u_h)\right|_{t=0} &=& (u_{0,h}, u_{1,h}).
\end{eqnarray}
Here we assume that $W\colon B_r(I)\to \R$ is a smooth function for some $r>0$ which is frame invariant, i.e., $W(RF)=W(F)$ for every $F\in \R^{d\times d}$ and $R\in SO(d)$, and such that $DW(I)=0$ and $D^2W(I)\colon \R^{d\times d}\to\R^{d\times d}$ is positive definite on symmetric matrices. -- Note that the latter condition implies that $D^2W(I)$ is elliptic in the sense of Legendre-Hadamard:
\begin{equation}\label{eq:LegendreHadamard}
  (D^2W(I)a\otimes b) :a\otimes b\geq c_0 |a|^2|b|^2\qquad \text{for all}\ a,b\in \R^d
\end{equation}
for some $c_0>0$. In the following, we will denote $z=(t,x')$ with the convention that $z_0=t$ and $z_j=x_j$ for $j=1,\ldots, d-1$ and $\nabla_z=\nabla_{t,x'}= (\partial_t, \nabla_{x'})$.

Our main result is:
\begin{theorem}\label{thm:main1}
  Let $\theta\geq 0$, $0<T<\infty$, let $f_h\in W^3_1(0,T;L^2)\cap W^1_1(0,T;H^{2}_{per})$, $0<h\leq 1$, 
and 
let $u_{0,h}\in H^4_{per}(\Omega)^d$, $u_{1,h}\in H^{3}_{per}(\Omega)^d$
  such that 
  \begin{eqnarray*}
    D\widetilde{W}(\nabla_h u_{0,h})e_d|_{x_d=\pm\frac12} = D^2\widetilde{W}(\nabla_h u_{0,h})\nabla_h u_{1,h} e_d|_{x_d=\pm\frac12} &=&0,\\
\left(D^2W(\nabla_h u_{0,h})\nabla_h u_{2,h} + D^3W(\nabla_h
   u_{0,h})[\nabla_h u_{1,h}, \nabla_h u_{1,h}]\right)e_d|_{x_d=\pm \frac12} &=&0,
  \end{eqnarray*}
 and    
  \begin{eqnarray}\label{eq:AssumInitialData1}
   \left\|\frac1h \eps_h (u_{0,h})\right\|_{H^1}+\max_{k=0,1,2}\left\|\left(\frac1h \eps_h (u_{1+k,h}),u_{2+k,h}\right)\right\|_{H^{2-k,0}} &\leq& Mh^{1+\theta}, \\\label{eq:AssumFData}
 \max_{|\gamma|\leq 2}\left( \|\partial_z^\gamma f_h|_{t=0}\|_{L^2} + \| \partial_z^\gamma f_h\|_{W^1_1(0,T;L^2)}\right) &\leq & M,  
  \end{eqnarray}
uniformly in $0<h\leq 1$, 
  where 
  \begin{eqnarray}\label{eq:u2h}
    u_{2,h} &=&  h^{1+\theta}f_h|_{t=0} + \frac1{h^2} \Div_h D\widetilde{W}(\nabla_h
  u_{0,h}),\\\label{eq:u3h}
    u_{3,h} &=&  h^{1+\theta}\partial_tf_h|_{t=0} + \frac1{h^2} \Div_h \left(D^2\widetilde{W}(\nabla_h
  u_{0,h})\nabla_h u_{1,h}\right),
 \\\nonumber
     u_{4,h} &=&  h^{1+\theta}\partial_t^2f_h|_{t=0} + \frac1{h^2} \Div_h D^2W(\nabla_h
   u_{0,h})\nabla_h u_{2,h}\\\label{eq:u4h}
 && + \frac1{h^2} \Div_h D^3W(\nabla_h
   u_{0,h})[\nabla_h u_{1,h}, \nabla_h u_{1,h}].
  \end{eqnarray}
 If $\theta>0$, then there is some
  $h_0\in (0,1]$ and $C$ depending on $M$ and $T$ such that for every $0<h\leq h_0$ there is a unique solution $u_h\in C^4([0,T];L^2)\cap C^0([0,T];H^{4}_{per})$ of
  (\ref{eq:Evol1})-(\ref{eq:BC3}) satisfying 
 \begin{eqnarray}\label{eq:UnifBdd1} 
   \max_{|\gamma|\leq 2}\left\|\left(\partial_t^2\partial_z^\gamma u_h,
       \nabla_{x,t}\partial_{z}^\gamma\frac1h\eps_h(u_h)\right)\right\|_{C([0,T];L^2)}&\leq& Ch^{1+\theta}
  \end{eqnarray}
uniformly in $0<h\leq h_0$. If $\theta=0$, the same statement holds with $h_0=1$ provided that $M$ is sufficiently small.
\end{theorem}

As mentioned before, for any fixed $h>0$ existence of a solution $u_h$ in the function spaces above is essentially known if $[0,T]$ above is replaced by some $[0,T'(h)]$, $T'(h)>0$. This follows from the result and arguments in \cite{KochWaveEquation}. More precisely, we have:
\begin{theorem}\label{thm:ShortTimeExistence}
  Let the assumptions of Theorem~\ref{thm:main1} be valid. Then there is a neighborhood $U_h$ of $0$ in $H^4(\Omega)^d$ such that for any $0<h\leq 1$ there is some $0<T_{max}(h)\leq \infty$ such that
  (\ref{eq:Evol1})-(\ref{eq:BC3}) has a  unique solution $u_h\in C^4([0,T);L^2)\cap C^0([0,T);H^{4})$. If $T_{max}(h)<\infty$, then either $\{\nabla u_h(t):t\in [0,T_{max}(h))\}$ is not precompact in $U_h$ or 
  \begin{equation}\label{eq:BlowUp}
    \lim_{t\to T_{max}(h)} \int_0^{t} \|\nabla_{x,t}^2 u(s)\|_{L^\infty(\Omega)}\sd s =\infty. 
  \end{equation}
\end{theorem}
\begin{remark}\label{rem:Neighborhood}
Here the neighborhood $U_h$ can be chosen as 
$$
U_h=\left\{u\in H^4_{per}(\Omega)^d: \left\|\left(\tfrac1h \eps_h (u),\nabla_h u\right)\right\|_{L^\infty}\leq \eps h \right\},
$$ 
where $\eps$ is so small that $\widetilde{W}\in C^\infty(\overline{B_\eps(0)})$ and the coercivity estimate \eqref{eq:UniformCoercive'} below holds.  
\end{remark}
We refer to the appendix for comments on the proof.

Because of Theorem~\ref{thm:ShortTimeExistence}, it only remains to show the uniform estimate \eqref{eq:UnifBdd1} in Theorem~\ref{thm:main1}. To this end  suitable $h$-independent estimates for the linearized system will be an important ingredient. This is the purpose of the following section.

\subsection{Estimates for the Linearized Operator}\label{sec:Linear}

Recall that $z=(t,x')$ with the convention that $z_0=t$ and $z_j=x_j$ for $j=1,\ldots, d-1$ and $\nabla_z=\nabla_{t,x'}= (\partial_t, \nabla_{x'})$.

Let $u_h$ for some $0<h\leq 1$ be given  such that
 \begin{eqnarray}\label{eq:UniformEstim} 
   \max_{|\gamma|\leq 2}\left\|\left(\frac1h\eps_h(\partial_{z}^\gamma u_h),
       \nabla_{x,t} \frac1h\eps_h(\partial_{z}^\gamma u_h)\right)\right\|_{C([0,T];L^2)}&\leq& Rh
  \end{eqnarray}
where $R\in (0,R_0]$ for some $0<R_0\leq 1$ to be determined later. 
 For the following we denote
\begin{equation*}
      \|f\|_{V_h}= \|(f,\nabla_h f)\|_{H^{1,0}},\quad \|g\|_{1,h}= \|(g,\nabla_h g)\|_{L^2},
\end{equation*}
where $f\in V(\Omega)$, $g\in H^1(\Omega)$. Of course $\|f\|_{V}\leq \|f\|_{V_h}$ and $\|g\|_{H^1}\leq \|g\|_{1,h}$ for all $0<h\leq 1$.

We note that (\ref{eq:UniformEstim}) and Korn's inequality (\ref{eq:Korn})
 imply
\begin{align}\nonumber 
  \max_{|\gamma|\leq 1}&\left\|\left(\partial_z^\gamma  \nabla_h u_h, \partial_z^\gamma \frac1h \eps_h (u_h) \right)\right\|_{C([0,T];V)}+ \max_{|\gamma|\leq 2}\left\|\left(\partial_z^\gamma  \nabla_h u_h, \partial_z^\gamma \frac1h \eps_h (u_h) \right)\right\|_{C([0,T];H^1)}\\\label{eq:LInftyUnifEstim}
& + \max_{|\gamma|\leq 3}\left\|\left(\partial_z^\gamma  \nabla_h u_h, \partial_z^\gamma \frac1h \eps_h (u_h) \right)\right\|_{C([0,T];L^2)} \leq C_1 R h
\end{align}
for some $C_1\geq 1$ depending only on the constant in the Korn inequality.
 Because of $V(\Omega)\hookrightarrow L^\infty(\Omega)$, cf. Lemma~\ref{lem:Embeddings},  
(\ref{eq:LInftyUnifEstim}) implies in particular
\begin{eqnarray}
  \label{eq:CoeffLInftyEstim}
  \left\|\nabla_h u_h\right\|_{C([0,T];V_h)}+
  \left\|\left( \nabla_h u_h, \frac1h\eps_h(
      u_h)\right)\right\|_{C([0,T];L^\infty\cap V)}\leq MRh, 
\end{eqnarray}
where $M$ depends only on $\Omega$. Here we have used that $\|\nabla_h^2u\|_{L^2(\Omega)}\leq C\left\|\nabla \frac1h\eps_h(u) \right\|_{L^2(\Omega)}$ due to Korn's inequality.  
Recall that $\widetilde{W}(A)=W(I+A)$ for all $A\in \R^{d\times d}$.
In order to evaluate $D\widetilde{W}( \nabla_hu_h)$,
we will assume that $R_0>0$ is so small that $\widetilde{W}\in C^\infty(\ol{B_{MR_0}(0)})$ and $MR_0 \leq \eps$, where $\eps>0$ is as in Corollary~\ref{cor:D3W}.

Using (\ref{eq:CoeffLInftyEstim}) and (\ref{eq:D3WEstim}), we obtain
\begin{eqnarray}\nonumber
 \lefteqn{\left| \frac1{h^2}\int_0^1\left(D^3\widetilde{W}(\tau \nabla_h u_h(t))[\nabla_h u_h(t),\nabla_h
   v], \nabla_h w\right)_{L^2(\Omega)}\sd \tau\right|}\\\nonumber
&\leq& C'_0\frac1h
   \left\|\left(\nabla_h u_h, \frac1h \eps_h (u_h)\right)\right\|_{V(\Omega)}\left\|\frac1h
   \eps_h(v)\right\|_{L^2(\Omega)}\left\|\frac1h
   \eps_h(w)\right\|_{L^2(\Omega)}\\\label{eq:EstimD3W}
 &\leq & C_0R \left\|\frac1h
   \eps_h(v)\right\|_{L^2(\Omega)}\left\|\frac1h
   \eps_h(w)\right\|_{L^2(\Omega)}
\end{eqnarray}
uniformly in $v,w\in H^1_{per}(\Omega)^d$, $0\leq t\leq T$, $0<h\leq 1$.

In particular, we derive
 \begin{eqnarray*}
   \lefteqn{\frac1{h^2}(D^2\widetilde{W}(\nabla_hu_h(t))\nabla_h v,\nabla_h v)_{L^2(\Omega)}=\frac1{h^2}(D^2\widetilde{W}(0)\nabla_h v,\nabla_h v)_{L^2(\Omega)}}\\
   && +
   \frac1{h^2}\int_0^1\left( D^3\widetilde{W}(\tau  \nabla_h u_h(t))[\nabla_h u_h(t),\nabla_h
   v], \nabla_h v\right)_{L^2(\Omega)}\sd \tau\\
 &\geq & 
   \left(c_0 - C_0R_0\right)  \left\|\frac1h
   \eps_h(v)\right\|_{L^2(\Omega)}^2,
 \end{eqnarray*}
uniformly in $v\in H^1_{per}(\Omega)^d$, $t\in [0,T]$, $0<T<\infty$,
 $0<R\leq R_0$, $0<h\leq 1$, where $c_0>0$ depends only on $D^2\widetilde{W}(0)$ and $\Omega$.
Hence, if $R_0\in (0,1]$ is sufficiently small,  we have
\begin{equation}\label{eq:UniformCoercive'}
  \frac1{h^2}(D^2\widetilde{W}(\nabla_hu_h(t))\nabla_h v,\nabla_h v)_{L^2(\Omega)} \geq \frac{c_0}2 \left\|\frac1h \eps_h(v)\right\|_{L^2(\Omega)}^2
\end{equation}
for all $v\in H^1_{per}(\Omega)^d$, $t\in[0,T]$, $0<h\leq 1$, $0<R\leq R_0$, and $u_h$ satisfying (\ref{eq:LInftyUnifEstim}), where $c_0$ is as above and depends only on $D^2\widetilde{W}(0)$ and $\Omega$. -- We note that the same conclusion holds if $\left\|\left(\tfrac1h \eps_h(u_h(t)),\nabla_h u_h(t)\right)\right\|_{L^\infty(\Omega)}\leq \eps h$ for $\eps>0$ sufficiently small. In particular, if $R_0>0$ is chosen sufficiently small, (\ref{eq:LInftyUnifEstim}) implies the latter condition. Hence, if $U_h$ is as in Remark~\ref{rem:Neighborhood}, \eqref{eq:UniformCoercive'} holds for every $u_h(t)\in U_h$. 

By the same kind of
expansion for $D^2\widetilde{W}$ and estimates one shows 
\begin{equation}\label{eq:EstimD2W}
  \left|\frac1{h^2}(\partial_{z_j}D^2\widetilde{W}(\nabla_hu_h(t))\nabla_h v,\nabla_h
    w)_{L^2(\Omega)}\right|
\leq C'R\left\|\frac1h
   \eps_h(v)\right\|_{L^2}\left\|\frac1h
   \eps_h(w)\right\|_{L^2}
\end{equation}
for all $v,w\in H^1_{per}(\Omega)^d$, $j=0,\ldots, d-1$ uniformly in $0<h\leq 1$, $t\in[0,T]$, $0<R\leq R_0$, $0<T<\infty$. 

\begin{remark}
  We note that a similar coerciveness estimate plays an important role in \cite{MonneauJustification}, where the stationary setting is considered. But there a scaling, which scales $v'(x)$ and $v_d(x)$ differently, is used.
\end{remark}

To obtain higher regularity, we will use:
\begin{lemma}\label{lem:LinEstim2} 
  Let $k=0,1$.
  There are constants $C_0>0,R_0\in (0,1]$ independent of  $R\in (0,R_0]$  such that, if $w\in
  H^2(\Omega)^d$ with $\nabla_{x'} w\in H^2(\Omega)$ if $k=1$  solves
  \begin{eqnarray*}
    - \frac1{h^2}\Div_h (D^2\widetilde{W}(\nabla_h u_h(t))\nabla_h w) &=& f  \qquad
    \text{in}\ \mathcal{D}'(\Omega)
  \end{eqnarray*}
  for some $f\in H^{k,0}(\Omega)$, $t\in [0,T]$, and $0<h\leq 1$ and $\nabla_h u_h$ satisfies (\ref{eq:CoeffLInftyEstim}) for $0<R\leq R_0$,  then
   we have
   \begin{equation}    \left\| \left(
       \nabla \frac1h \eps_h (w),\nabla_h^2 w\right)\right\|_{H^{k,0}(\Omega)}\label{eq:LinEstim2} 
 \leq C_0 \left( \|h^2f\|_{H^{k,0}(\Omega)} + \left\|\frac1h\eps_h (w)\right\|_{H^{1+k,0}(\Omega)}\right).
   \end{equation} 
   If additionally
  \begin{equation}
    \label{eq:BC}
 e_d\cdot D^2\widetilde{W}(\nabla_h u_h(t))\nabla_h
  w|_{x_d=\pm\frac12}=0,
  \end{equation}
 then
  \begin{equation}    \max_{j=0,1}\left\|\left(
      \nabla_h^{1+j} w, \nabla^j \frac1h \eps_h(w) \right)\right\|_{H^{k,0}(\Omega)}\label{eq:LinEstim2b} 
\leq C_0 \|f\|_{H^{k,0}(\Omega)}.
  \end{equation} 
\end{lemma}
\begin{proof} 
Let $0<R_0 \leq 1$ be at least as small as above.
  First of all, 
  \begin{eqnarray*}
    \lefteqn{\Div_h (D^2\widetilde{W}(0)\nabla_h w)
= \frac1h\partial_{x_d} (D^2\widetilde{W}(0)\nabla_h w)_d + \Div_{x'}(D^2\widetilde{W}(0)\nabla_h w)'}\\
&=& \frac1{h^2}(D^2\widetilde{W}(0)\partial_{x_d}^2 w\otimes e_d )_d + \frac1h (D^2\widetilde{W}(0)\partial_{x_d}(\nabla_{x'},0) w)_d + \Div_{x'}(D^2\widetilde{W}(0)\nabla_h w)' 
  \end{eqnarray*}
  where $A'=(a_{ij})_{i=1,\ldots d,j=1,\ldots d-1}$ for $A\in \R^{d\times d}$. We note that the second and third term consists of terms of
  $\nabla_{x'}\nabla_h w$. Moreover, 
  \begin{equation*}
   (D^2\widetilde{W}(0)\partial_{x_d}^2 w\otimes e_d )_d =M\partial_{x_d}^2 w  
  \end{equation*}
for some symmetric positive definite matrix $M$, which follows from the Legendre-Hadamard condition (\ref{eq:LegendreHadamard}).
  Hence
  \begin{equation*}
    \frac1{h^2} \partial_{x_d}^2 w =M^{-1}\left(\Div_h\left(Q\nabla_h w\right)  -\frac1h (Q\partial_{x_d}(\nabla_{x'},0) w)_d - \Div_{x'}(Q\nabla_h w)'\right)
  \end{equation*}
for $Q= D^2\widetilde{W}(0)$
and therefore
  \begin{eqnarray*}
     \lefteqn{\left\|
      \frac1{h^2} \partial_{x_d}^2 w\right\|_{H^{k,0}(\Omega)}}\\ 
    &\leq& C_0 \left(\left\|\Div_h \left(D^2\widetilde{W}(0)\nabla_h w\right) \right\|_{H^{k,0}(\Omega)}
    + \left\|\nabla_{x'}\nabla_h w\right\|_{H^{k,0}(\Omega)}\right).
  \end{eqnarray*}
  Thus Korn's inequality and $\|\partial_{x_d}\frac1h\eps_h(w)\|_{H^{k,0}(\Omega)}\leq \|\nabla_h^2 w\|_{H^{k,0}(\Omega)}$ yield
  \begin{eqnarray}\nonumber
      \lefteqn{\left\|\left(\nabla \frac1h \eps_h (w),
      \nabla_h^2  w\right)\right\|_{H^{k,0}(\Omega)}}\\\label{eq:nablah2}
    &\leq& C_0 \left(\left\|\Div_h \left(D^2\widetilde{W}(0)\nabla_h w\right) \right\|_{H^{k,0}(\Omega)}
    + \left\|\nabla_{x'}\frac1h\eps_h(w)\right\|_{H^{k,0}(\Omega)}\right).
  \end{eqnarray}
  Next we use that
  \begin{eqnarray*}
   \lefteqn{\Div_h \left(D^2\widetilde{W}(\nabla_h u_h)\nabla_h w\right)}\\
 &=&    \Div_h
    \left(D^2\widetilde{W}(0)\nabla_h w\right) +    \int_0^1\Div_h 
    \left(D^3\widetilde{W}(\tau \nabla_h u_h)[\nabla_hu_h,\nabla_h w]\right)\sd \tau \\
 &\equiv &    \Div_h
    \left(D^2\widetilde{W}(0)\nabla_h w\right) +    \Div_h 
    \left(G(\nabla_h u_h)[\nabla_hu_h,\nabla_h w]\right), 
  \end{eqnarray*}
  where $G\in C^\infty(\ol{B_{\eps}(0)};\mathcal{L}^3(\R^{d\times d}))$ for some suitable $\eps>0$. Hence, if $k=0$, Corollary~\ref{cor:Algebra} implies
  \begin{eqnarray*}
    \lefteqn{\|G(\nabla_h u_h)[\nabla_hu_h,\nabla_h w]\|_{1,h}}\\
      &\leq& C\|G(\nabla_h u_h)\|_{V_h}\|\nabla_hu_h\|_{V_h}\|\nabla_hw\|_{1,h}
      \leq CR_0\left\|\left(\nabla_hw, \nabla_h^2 w\right)\right\|_{L^2}.
    \end{eqnarray*}
    where 
$\|f\|_{V_h}= \|(f,\nabla_h f)\|_{H^{1,0}}$
    and we have used \eqref{eq:CoeffLInftyEstim}. Similarly, if $k=1$,
Corollary~\ref{cor:Algebra} yields
  \begin{eqnarray*}
    \lefteqn{\|G(P_n^h\nabla_h u_h)[\nabla_hu_h,\nabla_h w]\|_{V_h}}\\
      &\leq& C\|G(\nabla_h u_h)\|_{V_h}\|\nabla_hu_h\|_{V_h}\|\nabla_hw\|_{V_h}
      \leq CR_0\left\|\left(\nabla_hw, \nabla_h^2 w\right)\right\|_{H^{1,0}}.
    \end{eqnarray*}
  Therefore
  \begin{eqnarray}\nonumber
    \lefteqn{\left\|\Div_h \left(D^2\widetilde{W}(0)\nabla_h w\right)
      \right\|_{H^{k,0}(\Omega)}}\\\nonumber
   &\leq & \left\|\Div_h \left(D^2\widetilde{W}(\nabla u_h)\nabla_h
       w\right)\right\|_{H^{k,0}(\Omega)}\\\nonumber 
&& +\left\|\nabla_h (G(\nabla_h u_h)[\nabla_hu_h,\nabla_h w]\right\|_{H^{k,0}(\Omega)}\\ \label{eq:deltaEstim}
   &\leq & \left\|h^2f\right\|_{H^{k,0}(\Omega)} +  CR_0\left\|\left(\frac1{h}\eps_h(w), \nabla_h^2w\right)\right\|_{H^{k,0}(\Omega)}
  \end{eqnarray}
  for $k=0,1$.
Combining the last estimate with (\ref{eq:nablah2}) for sufficiently small $R_0 \in (0,1]$,  we obtain (\ref{eq:LinEstim2}). 

Now, if additionally (\ref{eq:BC}) holds, then
\begin{equation*}
  \frac1{h^2}(D^2\widetilde{W}(\nabla_h u_h)\nabla_h  w, \nabla_h \varphi)_{L^2} = (f,\varphi)_{L^2}
\end{equation*}
for all $\varphi \in H^1_{per}(\Omega)^d$. Hence, choosing $\varphi = \partial_{x'}^{2\gamma}\partial_{x'}^{2\beta} w_0$ with $w_0= w- \frac1{|\Omega|} \int_{\Omega} w\sd x$ and $|\beta|\leq k$, $|\gamma|\leq 1$ and using integration by parts, we obtain by (\ref{eq:UniformCoercive'}), (\ref{eq:EstimD2W}), (\ref{eq:LinEstim2}), and \eqref{eq:D1EstimAn} below
\begin{eqnarray*}
  \lefteqn{\sup_{|\beta|\leq k, |\gamma|\leq 1}\left\|\partial_{x'}^\gamma \partial_{x'}^\beta \frac1h \eps_h(w) \right\|_{L^2(\Omega)}^2}\\
 &\leq& C_0 \|f\|_{H^{k,0}(\Omega)}\max_{|\gamma|\leq 1}\left\|\partial_{x'}^{2\gamma} w_0\right\|_{H^{k,0}(\Omega)}
+ CR \left\|\frac1h \eps_h(w)\right\|_{H^k(\Omega)}\max_{|\gamma|\leq 1}\left\|\partial_{x'}^{\gamma} w_0\right\|_{H^{k,0}(\Omega)}\\
&\leq& C' \left(\|f\|_{H^{k,0}(\Omega)}+R\left\|\frac1h \eps_h (w)\right\|_{H^{1,0}(\Omega)}\right)\max_{|\gamma|\leq 1}\left\|\partial_{x'}^{\gamma} \frac1h \eps_h (w)\right\|_{H^{k,0}(\Omega)}.
\end{eqnarray*}
Thus, choosing $R_0$ sufficiently small, we obtain
\begin{equation*}
  \left\| \frac1h \eps_h(w) \right\|_{H^{1+k,0}(\Omega)} \leq C_0 \|f\|_{H^{k,0}(\Omega)}
\end{equation*}
with $C_0>0$ depending only on $\Omega$. This finishes the proof.
 \end{proof}

 \begin{lemma}\label{lem:DifEstimAn}
Let $\nabla u_h(t)$ satisfy (\ref{eq:LInftyUnifEstim}) for some $0<h\leq 1$, $t\in [0,T]$, and $0<R\leq R_0$, where $R_0\in (0,1]$ is so small that all previous conditions are satisfied.  
Then
   \begin{eqnarray}\label{eq:D1EstimAn}
     \left|\frac1{h^2}\left((\partial_z^\beta D^2\widetilde{W}(\nabla_hu_h(t)))\nabla_h
         w,\nabla_h v\right)_\Omega\right|
 \leq
     CR\left\|\frac1h\eps_h(w)\right\|_{H^{|\beta|-1}(\Omega)}\left\|\frac1h\eps_h(v)\right\|_{L^2(\Omega)}
   \end{eqnarray}
   if $1\leq |\beta|\leq 2$ and 
   \begin{eqnarray}\label{eq:3DiffEstim}
     \left|\frac1{h^2}\left((\partial_z^\beta D^2\widetilde{W}(\nabla_hu_h(t)))\nabla_h
         w,\nabla_h v\right)_\Omega\right|
 \leq
     CR\left\| \frac1h\eps_h(w)\right\|_{V(\Omega)} \left\|\frac1h\eps_h(v)\right\|_{L^2(\Omega)}
   \end{eqnarray}
if $|\beta|=3$.
The constants $C$ are independent of $\nabla_h u_h(t), w,v, h,n, R$.
 \end{lemma}
 \begin{proof}
If $|\beta|=1$, then (\ref{eq:D1EstimAn}) is just (\ref{eq:EstimD2W}).
Next let $|\beta|=2$.
  Then for $j,k=0,\ldots, d-1$
  \begin{eqnarray*}
    \partial_{z_j}\partial_{z_k} D^2\widetilde{W}(\nabla_h u_h) &=& D^3\widetilde{W}(\nabla_h
    u_h)[\partial_{z_j}\partial_{z_k}\nabla_h u_h]\\
&& + D^4 \widetilde{W}(\nabla_h
    u_h)[\partial_{z_j}\nabla_h u_h,\partial_{z_k}\nabla_h u_h],
  \end{eqnarray*}
  where 
  \begin{eqnarray*}
   \left\|\partial_{z_j}\partial_{z_k} \frac1h \eps_h (u_h)\right\|_{H^{1,0}(\Omega)}
&\leq& 
    C_0 \left\| \nabla_z \frac1h \eps_h (u_h)\right\|_{V(\Omega)} \leq C_0Rh 
\end{eqnarray*}
due to (\ref{eq:LInftyUnifEstim}). Together with (\ref{eq:D3WEstimb}) the latter estimate implies (\ref{eq:D1EstimAn}) in the case $|\beta|=2$. 

  Finally, if $|\beta|=3$, we use that
  \begin{eqnarray*}
    \lefteqn{\partial_{z_j}\partial_{z_k}\partial_{z_l} D^2\widetilde{W}(\nabla_h u_h) = D^3\widetilde{W}(\nabla_h
    u_h)[\partial_{z_j}\partial_{z_k}\partial_{z_l}\nabla_h u_h]}\\
&& + D^4 \widetilde{W}(\nabla_h
    u_h)[\partial_{z_j}\partial_{z_l}\nabla_h u_h,\partial_{z_k}\nabla_h u_h]\\
  &&  +  D^4 \widetilde{W}(\nabla_h
    u_h)[\partial_{z_j}\nabla_h u_h,\partial_{z_k}\partial_{z_l}\nabla_h u_h]\\
&& +  D^4 \widetilde{W}(\nabla_h
    u_h)[\partial_{z_l}\nabla_h u_h,\partial_{z_j}\partial_{z_k}\nabla_h u_h]\\
&&
   + D^5W (\nabla_hu_h) [\partial_{z_j}\nabla_h u_h,\partial_{z_k}\nabla_h u_h,\partial_{z_l}\nabla_h u_h]
  \end{eqnarray*}
  Since $\nabla_z \nabla_h u_h\in L^\infty(\Omega)$  and $\nabla_z^2 \nabla_h u_h\in H^{1,0}(\Omega) \hookrightarrow L^{4,2}(\Omega)$ are of order $CRh$ due to (\ref{eq:LInftyUnifEstim}), the estimates of all parts in
  \begin{equation*}
    \frac1{h^2} \left((\partial_z^\beta D^2\widetilde{W}(\nabla_h u_h))\nabla_h w, \nabla_h v\right)_\Omega
  \end{equation*}
which come from terms involving $D^4\widetilde{W}$ or $D^5\widetilde{W}$ can be done in a straight forward manner by
\begin{equation*}
  CR \left\| \frac1h \eps_h (w)\right\|_{L^{4,\infty}}\left\| \frac1h \eps_h (v)\right\|_{L^2}\leq   C'R \left\| \frac1h \eps_h (w)\right\|_{H^{1}}\left\| \frac1h \eps_h (v)\right\|_{L^2}
\end{equation*}
uniformly in $0<h\leq 1$ and $n\in \N_0\cup\{\infty\}$.
 It only remains to estimate the part involving the $D^3\widetilde{W}$-term:
To this end we use that (\ref{eq:LInftyUnifEstim}) and (\ref{eq:D3WEstim}) imply
\begin{eqnarray*}
     \lefteqn{\left|\frac1{h^2}\left(( D^3\widetilde{W}(\nabla_hu_h))[\partial_z^\beta \nabla_hu_h,\nabla_h
         w],\nabla_h v\right)_\Omega\right|} \\
   &\leq& \frac{C_0}h \left\|\partial_z^\beta\frac1h\eps_h(u_h)  \right\|_{L^2(\Omega)} \left\| \left( \frac1h \eps_h(w), \nabla_h w \right) \right\|_{V(\Omega)}\left\| \frac1h \eps_h(v)  \right\|_{L^2(\Omega)}\\
   &\leq& CR \left\| \frac1h \eps_h(w) \right\|_{V(\Omega)}\left\| \frac1h \eps_h(v)  \right\|_{L^2(\Omega)}
\end{eqnarray*}
Altogether we obtain (\ref{eq:3DiffEstim}). 
 \end{proof}

Next we consider the linearized system to (\ref{eq:Evol1})-(\ref{eq:BC3}):
\begin{alignat}{1}\label{eq:stEq}
  \partial_t^2 w - \frac1{h^2} \Div_h (D^2\widetilde{W} (\nabla_h
  u_h)\nabla_h w) &= f\\\label{eq:Neumann1}
   D^2\widetilde{W} (\nabla_h
  u_h)\nabla_h w\, e_d|_{x_d=\pm\frac12} &=0 \\\label{eq:Periodic1}
w &\ \ \text{is}\ 2L\text{-periodic w.r.t.}\ x_j, j=1,d-1,\\\label{eq:Initial}
  \left. (w,\partial_t w)\right|_{t=0} &= (w_0, w_1).
\end{alignat}
The following lemma contains the essential estimate for this system. 
\begin{lemma}\label{lem:LinEstim3} 
  Let $0<T<\infty$, $0<h\leq 1$, $0<R\leq R_0$ be given, and let $R_0$ be as in Lemma~\ref{lem:LinEstim2}. 
  Moreover, assume that $u_h$ satisfies (\ref{eq:LInftyUnifEstim}).
  \begin{enumerate}
  \item For every $f\in W^1_1(0,T;L^2)^d$, $w_0\in H^2_{per}(\Omega)^d$, and $w_1 \in H^1_{per}(\Omega)^d$,
  there is a unique $w\in C^0([0,T];H^2_{per}(\Omega))^d\cap C^2([0,T];L^2(\Omega))^d$  that solves \eqref{eq:stEq}-\eqref{eq:Periodic1}.
Moreover, there are constants $C_L, C'\geq 1$ depending only on $\Omega$ and $W$  such that
  \begin{eqnarray}\label{eq:LinEstim'}
    \lefteqn{\left\|\left(\partial_t^2 
          w,\frac1h \eps_h(w), \nabla_{x,t}\frac1h \eps_h(w)\right)\right\|_{C([0,T];L^2)}}\\\nonumber
&\leq& C_Le^{C'RT} \left( 
      \| f\|_{W^1_1(0,T;L^2)}
 + \left\|\left(\frac1h \eps_h(w_1), w_2, f|_{t=0}\right)\right\|_{L^2}  \right)
  \end{eqnarray}
  where
  \begin{eqnarray}\label{eq:LinEstimw2}
    w_{2}&=& \frac1{h^2}\Div_h (D^2\widetilde{W}(\nabla_h
u_h|_{t=0})\nabla_h w_0)+   f|_{t=0}.
  \end{eqnarray}
 \item For every $f\in W^2_1(0,T;L^2)^d\cap W^1_1(0,T;H^1_{per})^d$, $w_0\in H^3_{per}(\Omega)^d$, and $w_1 \in H^2_{per}(\Omega)^d$,
  there is a unique $w\in C^0([0,T];H^3_{per}(\Omega))^d\cap C^3([0,T];L^2(\Omega))^d$  that solves \eqref{eq:stEq}-\eqref{eq:Periodic1}.
Moreover, there are  constants $C_L, C'\geq 1$ depending only on $\Omega$ and $W$  such that
  \begin{eqnarray}\nonumber
    \lefteqn{\max_{|\gamma|\leq 1}\left\|\left(\partial_t^2\partial_z^\gamma 
          w,\frac1h \eps_h(\partial_z^\gamma w), \nabla_{x,t}\frac1h \eps_h(\partial_z^\gamma w)\right)\right\|_{C([0,T];L^2)}}\\\nonumber
&\leq& C_Le^{C'RT} \left( 
      \max_{|\gamma|\leq 1}\| \partial_z^\gamma f\|_{W^1_1(0,T;L^2)}
 + \left\|\left(\frac1h \eps_h(w_1), w_2, f|_{t=0}\right)\right\|_{H^{1,0}}\right.\\\label{eq:LinEstim}
&& \left.+ \left\|\left(\frac1h \eps_h(w_2), w_3, \partial_tf|_{t=0}\right)\right\|_{L^2}  \right)
  \end{eqnarray}
  where $w_2$ is as above and 
  \begin{eqnarray*}
    w_{3}&=& \frac1{h^2}\Div_h (D^2\widetilde{W}(\nabla_h
u_h|_{t=0})\nabla_h w_1)+   \partial_tf|_{t=0}.
  \end{eqnarray*}
  \end{enumerate}
\end{lemma}
\begin{proof} 
In both parts existence of a solution (for fixed $h$) can be obtained  by the energy method as e.g. in \cite{KochWaveEquation}. 

Hence the main task is to  establish the uniform estimates (\ref{eq:LinEstim'}) and (\ref{eq:LinEstim}).
 First of all, we note that (\ref{eq:stEq})-(\ref{eq:Periodic1}) imply 
 \begin{equation*}
   a(t):=\int_\Omega w(t)\sd x= \int_\Omega w_0 \sd x + t \int_\Omega w_1 \sd x + \int_0^t (t-\tau)\int_\Omega  f(\tau,x) \sd x \sd \tau.
 \end{equation*}
 Hence, replacing $w(t)$ by $w(t)-a(t)$ and subtracting from $(w_0,w_1,f)$ their mean values with respect to $\Omega$, we can reduce to the case
 \begin{equation*}
    \int_\Omega w_0 \sd x= \int_\Omega w_1 \sd x = \int_\Omega f(t) \sd x = \int_\Omega w(t)\sd x= 0
 \end{equation*}
 for all $0\leq t\leq T$.

Now we first prove \eqref{eq:LinEstim'}. To this end we
differentiate (\ref{eq:stEq}) with respect to $t$ and multiply with $\partial_t^2 w$ in
  $L^2(\Omega)$. 
Then we obtain
  \begin{eqnarray*}
    \lefteqn{\frac12\frac{d}{dt}\left( \left\|\partial_t^2 w\right\|_{L^2}^2+\frac1{h^2}\left(D^2\widetilde{W}(\nabla_hu_h)\nabla_h\partial_t w,
        \nabla_h \partial_t w\right)_{L^2}  \right)}\\
&\leq &  \left|\left(\partial_t f,\partial_t^2 w\right)_{L^2}\right|+ \frac32 \left|\frac1{h^2}\left( (\partial_t D^2\widetilde{W}(\nabla_h
  u_h))\nabla_h \partial_t w,\nabla_h \partial_t w\right)_{L^2}\right|\\
&&
+  \left|\frac1{h^2}\left( (\partial_t^2 D^2\widetilde{W}(\nabla_h
  u_h))\nabla_h w,\nabla_h \partial_t w\right)_{L^2}\right|
 -  \frac1{h^2}\frac{d}{dt}\left((\partial_t D^2\widetilde{W}(\nabla_hu_h))\nabla_h w,
        \nabla_h \partial_t w\right)_{L^2}
  \end{eqnarray*}
in the sense of distributions,
  where we have used 
  \begin{eqnarray}\nonumber
    \lefteqn{\frac{d}{dt}\left( \frac1{2h^2}\left(D^2\widetilde{W}(\nabla_hu_h)\nabla_h\partial_t w,
        \nabla_h \partial_t w\right)_{L^2} + \frac1{h^2}\left((\partial_t D^2\widetilde{W}(\nabla_hu_h))\nabla_h w,
        \nabla_h \partial_t w\right)_{L^2} \right) }\\\nonumber
    &=& -\frac12\frac{d}{dt} \|\partial_t^2 w\|_{L^2}^2 + \left(\partial_t f, \partial_t^2 w\right)_{L^2} + \frac3{2h^2} \left((\partial_t D^2\widetilde{W}(\nabla_hu_h))\nabla_h\partial_t w,
        \nabla_h \partial_t w\right)_{L^2}\\\label{eq:EnergyDiff}
&&+  \frac1{h^2}\left((\partial_t^2 D^2\widetilde{W}(\nabla_hu_h))\nabla_h w,
        \nabla_h \partial_t w\right)_{L^2}  
  \end{eqnarray}
  and (\ref{eq:Neumann1})-(\ref{eq:Periodic1}). Due to \eqref{eq:D1EstimAn} we have
  \begin{eqnarray*} 
   \frac1{h^2} \left|\left((\partial_t D^2\widetilde{W}(\nabla_hu_h))\nabla_h\partial_t w,
        \nabla_h \partial_t w\right)_{L^2}\right|
    &\leq& CR\left\|\frac1{h}\eps_h(\partial_t w)\right\|_{L^2}^2,\\
\frac1{h^2}\left|\left((\partial_t^2 D^2\widetilde{W}(\nabla_hu_h))\nabla_h w,
        \nabla_h \partial_t w\right)_{L^2}\right|
    &\leq& CR\left\|\frac1{h}\eps_h(w)\right\|_{H^1(\Omega)}\left\|\frac1{h}\eps_h(\partial_t w)\right\|_{L^2}
  \end{eqnarray*} 
for every $t\in [0,T]$.
  Moreover, because of \eqref{eq:D1EstimAn} again,
  \begin{eqnarray*}
    \lefteqn{\sup_{0\leq \tau \leq t}\left|\frac1{h^2}\left((\partial_t D^2\widetilde{W}(\nabla_hu_h(\tau)))\nabla_h w(\tau),
        \nabla_h \partial_t w(\tau)\right)_{L^2}\right|}\\
  &\leq & C R\left\| \frac1h \eps_h (w)\right\|_{L^\infty(0,t;L^2)}
\left\| \frac1h \eps_h (\partial_t w)\right\|_{L^\infty(0,t;L^2)}
  \end{eqnarray*}
  Therefore the previous estimates, (\ref{eq:UniformCoercive'}), and Young's inequality imply
  \begin{eqnarray}\nonumber
    \lefteqn{\sup_{0\leq \tau \leq t}\left\|\left(\partial_t^2 w(\tau),\frac1h \eps_h(\partial_t w(\tau))
        \right)\right\|_{L^2}^2  }\\
\nonumber
    &\leq & CR \left\|\left( \partial_t^2 w, \frac1h \eps_h(w),\nabla_{x,t}\frac1h\eps_h( w)\right)\right\|_{L^2(0,t;L^2)}^2+
  C_0\|\partial_t f\|_{L^1(0,T;L^2)}^2   \\\nonumber 
&& + C_0\left\|\left(\frac1h \eps_h(w_1),
        w_2\right)\right\|_{L^2}^2 + C R\left\| \frac1h \eps_h (w)\right\|_{C([0,t];L^2)}^2.
  \end{eqnarray}
  Now
  \begin{eqnarray*}
    \left\| \frac1h \eps_h (w)\right\|_{C([0,t];L^2)}^2 &\leq& C_0\left(\max_{j=0,1}\left\| \frac1h \eps_h (\partial_t^j w)\right\|_{L^2(0,t;L^2)}^2 + \left\|\frac1h \eps_h (w_0)\right\|_{L^2}^2 \right), 
    \end{eqnarray*}
due to 
\begin{equation}\label{eq:UnivLinftyEb}
\|f\|_{C([0,t];L^2)}\leq C_0 \left(\|f\|_{W^1_2(0,t;L^2)} + \|f|_{t=0}\|_{L^2}\right)  
\end{equation}
with some $C_0>0$ independent of $t>0$, cf. (\ref{eq:StrongBUCConv}), and
    \begin{eqnarray*}
   \lefteqn{\left\|\left(\frac1h \eps_h (w), \nabla \frac1h \eps_h (w)\right)\right\|_{L^\infty(0,t;L^2)}}\\
    &\leq& C_0 \left( \|f\|_{C([0,t];L^2)} + \|\partial_t^2 w\|_{C([0,t];L^2)}\right) \\
 &\leq& C_0 \left( \|f\|_{W^1_1(0,t;L^2)}+\|f|_{t=0}\|_{L^2} + \|\partial_t^2 w\|_{C([0,t];L^2)}\right) 
  \end{eqnarray*}
due to (\ref{eq:LinEstim2b}) 
and (\ref{eq:UnivLinftyEb})
  uniformly in $0\leq t\leq T$.
Hence we conclude
  \begin{eqnarray}\nonumber
    \lefteqn{\sup_{0\leq \tau \leq t}\left\|\left(\partial_t^2 w(\tau), \frac1h \eps_h(w),\nabla_{x,t}\frac1h \eps_h( w(\tau))
        \right)\right\|_{L^2}^2  }\\
\nonumber
    &\leq & CR \left\|\left( \partial_t^2 w,\frac1h \eps_h(w),\nabla_{x,t}\frac1h\eps_h(
        w)\right)\right\|_{L^2(0,t;L^2)}^2   \\\nonumber 
&& +
  C_0\| f\|_{W^1_1(0,T;L^2)}^2+ C_0\left\|\left(\frac1h \eps_h(w_1),
        w_2, f|_{t=0}\right)\right\|_{L^2}^2,
  \end{eqnarray}
  where we have used $R\leq 1$ and (\ref{eq:LinEstimw2}).
  Therefore the Lemma of Gronwall yields
  \begin{eqnarray*}
    \lefteqn{\left\|\left(\partial_t^2 w,\frac1h \eps_h(w), \nabla_{x,t}\frac1h \eps_h(w)
        \right)\right\|_{C([0,T];L^2)}^2}\\
    &\leq& C_Le^{C'RT}\left(\left\|\left(\frac1h \eps_h(w_1),
        w_2, f|_{t=0}\right)\right\|_{L^2}^2 + \|f\|_{W^1_1(0,T;L^2)}^2 \right).
  \end{eqnarray*}
  This shows (\ref{eq:LinEstim'}).

To prove \eqref{eq:LinEstim}, we differentiate (\ref{eq:stEq}) with respect to $z_j$, $j=0,\ldots, d-1$ and obtain that $\tilde{w}_j:= \partial_{z_j}w$ solves
\begin{eqnarray*}
    \partial_t^2 \tilde{w}_j - \frac1{h^2} \Div_h (D^2 \widetilde{W} (\nabla_h
  u_h)\nabla_h \tilde{w}_j) &=& f_j + \frac1{h^2} \Div_h ((\partial_{z_j}D^2 \widetilde{W} (\nabla_h
  u_h))\nabla_h w)
\end{eqnarray*}
together with 
\begin{equation*}
     D^2\widetilde{W} (\nabla_h
  u_h)\nabla_h \widetilde{w}_j\, e_d|_{x_d=\pm\frac12} =-    \left( \partial_{z_j}D^2\widetilde{W} (\nabla_h
  u_h)\right)\nabla_h w\, e_d|_{x_d=\pm\frac12}
\end{equation*}
and \eqref{eq:Periodic1}, where $f_j= \partial_{z_j} f$.
Hence differentiating again by $t$, multiplying this equation with $\partial_t^2 \tilde{w}_j$, and
applying the estimates above with $w$ replaced by $\tilde{w}_j$, we derive
  \begin{eqnarray*}
    \max_{|\gamma|\leq 1}\lefteqn{\left\|\left(\partial_t^2 \partial_z^\gamma w,\frac1h \eps_h(\partial_z^\gamma w), \nabla_{x,t}\frac1h \eps_h(\partial_z^\gamma w)
        \right)\right\|_{C([0,T];L^2)}^2}\\
    &\leq& C_Le^{C'RT}\max_{|\gamma|\leq 1}\left(\left\|\left(\frac1h \eps_h(\partial_z^\gamma w_1),
        \partial_z^\gamma w_2, \partial_z^\gamma f|_{t=0}\right)\right\|_{L^2}^2 + \|\partial_z^\gamma f\|_{W^1_1(0,T;L^2)}^2 \right)\\
&&+ \max_{j=0,\ldots, d-1}\left|\frac1{h^2}\left(\left(\partial_t^2 (\partial_{z_j} D^2\widetilde{W}(\nabla_h u_h))\nabla_h w\right),\partial_t \nabla_h \tilde{w}_j \right)_{Q_T}  \right|\\
&&+ \max_{j=0,\ldots, d-1}\left|\left.\frac1{h^2}\left(\left(\partial_t (\partial_{z_j} D^2\widetilde{W}(\nabla_h u_h(t)))\nabla_h w(t)\right),\partial_t \nabla_h \tilde{w}_j(t) \right)_{\Omega}\right|_{t=0}^T  \right|
  \end{eqnarray*}
  with the convention that $\partial_t w_j:= w_{j+1}$ and $Q_T=\Omega\times (0,T)$. Here  we have used that
  \begin{eqnarray*}
    -\lefteqn{\frac1{h^2}\left(\partial_t \Div_h (D^2\widetilde{W}(\nabla_h u_h)\nabla_h \tilde{w}_j)+\partial_t \Div_h\left( (\partial_{z_j} D^2\widetilde{W}(\nabla_h u_h))\nabla_h w\right),\partial_t^2  \tilde{w}_j \right)_{\Omega}}\\
 &=&\frac1{h^2}\left(\left(\partial_t (D^2\widetilde{W}(\nabla_h u_h(t)))\nabla_h \tilde{w}_j(t)\right),\partial_t^2 \nabla_h \tilde{w}_j(t) \right)_{\Omega}\\
&&-
    \frac1{h^2}\left(\left(\partial_t^2 (\partial_{z_j} D^2\widetilde{W}(\nabla_h u_h(t)))\nabla_h w(t)\right),\partial_t \nabla_h \tilde{w}_j(t) \right)_{\Omega}\\ 
&&+\frac{d}{dt}\frac1{h^2}\left(\left(\partial_t (\partial_{z_j} D^2\widetilde{W}(\nabla_h u_h(t)))\nabla_h w(t)\right),\partial_t \nabla_h \tilde{w}_j(t) \right)_{\Omega}
  \end{eqnarray*}
in the sense of $\mathcal{D}'(0,T)$.
Using 
\begin{eqnarray*}
\lefteqn{\partial_t^2\left[ (\partial_{z_j} D^2\widetilde{W}(\nabla_h u_h))\nabla_h w\right]= \left(\partial_{z_j} D^2\widetilde{W}(\nabla_h u_h)\right)\nabla_h \partial_t^2 w}\\
&&   
 + 
2\left(\partial_t\partial_{z_j} D^2\widetilde{W}(\nabla_h u_h)\right)\nabla_h\partial_t w+   
 \left(\partial_t^2\partial_{z_j} D^2\widetilde{W}(\nabla_h u_h)\right)\nabla_h w   
\end{eqnarray*}
we obtain with the aid of Lemma~\ref{lem:DifEstimAn}
\begin{eqnarray*}
    \lefteqn{\left|\frac1{h^2}\left(\partial_t^2\left( (\partial_{z_j} D^2\widetilde{W}(\nabla_h u_h))\nabla_h w\right),\partial_t \nabla_h \tilde{w}_j \right)_{Q_T}\right|}\\
    &\leq& CR\left(\left\| \frac1h\eps_h (\partial_t^2w)\right\|_{L^2(Q_T)}+ \left\| \frac1h\eps_h (\partial_tw)\right\|_{L^2(0,T;H^1)} + \left\| \frac1h\eps_h (w)\right\|_{L^2(0,T;V)}\right)\\
&&\times \left\| \frac1h\eps_h (\partial_t\tilde{w}_j)\right\|_{L^2(Q_T)}.
\end{eqnarray*}
Therefore this term can be absorbed in the left-hand side with the aid of the Lemma of Gronwall. Moreover,
\begin{eqnarray*}
  \lefteqn{\left|\left.\frac1{h^2}\left(\partial_t \left((\partial_{z_j} D^2\widetilde{W}(\nabla_h u_h))\nabla_h w\right),\partial_t \nabla_h \tilde{w}_j \right)_{\Omega}\right|_{t=0}^T\right|}\\
&\leq & CR \left(\left\| \frac1h\eps_h (\partial_tw)\right\|_{L^\infty(0,T;L^2)}+ \left\| \frac1h\eps_h (w)\right\|_{L^\infty(0,T;H^1)}\right) \left\| \frac1h\eps_h (\partial_t\partial_{z_j}w)\right\|_{L^\infty(0,T;L^2)},
\end{eqnarray*}
where the terms in $(\ldots)$ can be estimated by \eqref{eq:LinEstim'}. Thus applying Young's inequality this term can be absorbed too.

Combining the last estimates yields (\ref{eq:LinEstim}).
\end{proof}

Finally, we consider (\ref{eq:stEq})-(\ref{eq:Initial}) with $f$ replaced
by $-\Div_h f_1 + f_2$ 
in its weak form,
namely:
\begin{eqnarray}\nonumber
  -(\partial_t w,\partial_t \varphi)_{Q_T} &+& \frac1{h^2} (D^2\widetilde{W}(\nabla_h
  u_h)\nabla_h w,\nabla_h \varphi)_{Q_T}\\\label{eq:wEq1}
 &=& (f_1,\nabla_h \varphi)_{Q_T}+ (f_2,\varphi)_{Q_T} + \weight{w_1,\varphi|_{t=0}}_{W'_h,W_h} \\
w & \text{is}& 2L\text{-periodic w.r.t.}\ x_j, j=1,\ldots,d-1,
  \\\label{eq:wEq3}
  w|_{t=0} &=& w_0.
\end{eqnarray}
for all $\varphi \in C^1([0,T];H^1_{per}(\Omega)^d)$ with $\varphi|_{t=T}=0$, where $Q_T= \Omega\times (0,T)$.

Here and in the following we denote by $W_h(\Omega)$ the space $H^1_{per}(\Omega)^d\cap \{u: \int_\Omega u(x)\sd x=0\}$ equipped with the norm
\begin{equation*}
 \|u\|_{W_h(\Omega)}= \left\|\frac1h \eps_h(u)\right\|_{L^2(\Omega)}, \qquad u \in H^1(\Omega)^d
\end{equation*}
and $W_h'(\Omega)$ its dual space with norm
\begin{equation*}
  \|f\|_{W_h'(\Omega)}= \sup \left\{ \left|\weight{f,\varphi}_{W_h',W_h}\right|:u\in H^1_h(\Omega)\ \text{with}\ \left\|\frac1h\eps_h(u)\right\|_{L^2}=1\right\}.
\end{equation*}
Furthermore,
\begin{lemma}\label{lem:LinEstimWeak} 
Assume that $u_h$ satisfies
  (\ref{eq:LInftyUnifEstim}) with $R\in (0,R_0]$ and some given $0<h\leq 1$,  and   let $R_0\in (0,1]$ be so small that \eqref{eq:EstimD2W} and \eqref{eq:UniformCoercive'} hold.
  Let $w\in C^0([0,T];H^1(\Omega))^d\cap C^1([0,T];L^2(\Omega))^d$ be a solution of (\ref{eq:wEq1})-(\ref{eq:wEq3}) for some $f_1\in L^1(0,T;L^2(\Omega)^{d\times d})$, $f_2\in L^1(0,T;L^2(\Omega)^d)$, $w_0\in L^2(\Omega)^d$, and $w_1\in H^1_{per}(\Omega)^d$ and let $u(t)= \int_0^t w(\tau)\sd \tau$. Then there are some $C_0,C>0$  independent of $w$
  and $0<T<\infty$
  such that
  \begin{eqnarray}\nonumber
    \lefteqn{\left\|\left(w, \frac1h \eps_h (u)\right)\right\|_{C([0,T];L^2)}}\\\label{eq:WeakLinEstim}
    &\leq& C_0e^{CRT}\left(\|f_1\|_{L^1(0,T;L^2_h)} + \|f_2\|_{L^1(0,T;L^2)} +
      \left\|w_0\right\|_{L^2} + \|w_1\|_{W'_h(\Omega)}\right).
  \end{eqnarray} 
\end{lemma}
\begin{proof}
  Let $0\leq T'\leq T$ and define $\tilde{u}_{T'}(t)= -\int_t^{T'}w(\tau) \sd \tau$.
  We choose $\varphi= \tilde{u}_{T'}\chi_{[0,T']}$ in (\ref{eq:wEq1}) (after a standard approximation). Then
  \begin{eqnarray*}
    \lefteqn{\frac12\|w(T')\|_{L^2(\Omega)}^2  +\frac1{2h^2} (D^2\widetilde{W}(\nabla_h
  u_h)\nabla_h \tilde{u}_{T'}(0),\nabla_h \tilde{u}_{T'}(0))_{\Omega}}\\
&=& -\frac1{2h^2} ((\partial_t D^2\widetilde{W}(\nabla_h
  u_h))\nabla_h \tilde{u}_{T'},\nabla_h \tilde{u}_{T'})_{Q_{T'}} - (f_1,\nabla_h \tilde{u}_{T'})_{Q_{T'}}- (f_2,\tilde{u}_{T'})_{Q_{T'}}\\
&& - \weight{w_1, \tilde{u}_{T'}(0)}_{W'_h,W_h}  
  + \frac12\|w_0\|_{L^2(\Omega)}^2
  \end{eqnarray*}
Hence (\ref{eq:UniformCoercive'}), (\ref{eq:EstimD2W}), and $\tilde{u}_{T'}(0)= -u(T')$ imply
\begin{eqnarray*}
  \lefteqn{\|w(T')\|_{L^2(\Omega)}^2 + \left\|\frac1h\eps_h(u(T'))\right\|_{L^2}^2 \leq CR \left\|\frac1h\eps_h(\tilde{u}_{T'})\right\|_{L^2(Q_{T'})}^2+ C \|w_0\|_{L^2(\Omega)}^2}\\
    && 
+ C\left(\|f_1\|_{L^1(0,T;L^2_h)}+\|f_2\|_{L^1(0,T;L^2)}\|w_1\|_{W'_h}\right)\left\|\frac1h \eps_h (u) \right\|_{C([0,T'];L^2)}
\end{eqnarray*}
for all $0\leq T'\leq T$. Since $\tilde{u}_{T'}(t)= -u(T')+u(t)$, we obtain
\begin{equation*}
  \left\|\frac1h\eps_h(\tilde{u}_{T'})\right\|_{L^2(Q_{T'})}^2 \leq
\left\|\frac1h\eps_h(u)\right\|_{L^2(Q_{T'})}^2 + T'\left\| \frac1h\eps_h(u(T'))\right\|_{L^2(\Omega)}^2.
\end{equation*}
Hence there is some $\kappa>0$ independent of $R\in (0,R_0]$, $h\in (0,1]$, such that
\begin{eqnarray*}
  \lefteqn{\|w\|_{C([0,T'];L^2)}^2 + \left\|\frac1h\eps_h(u)\right\|_{C([0,T'];L^2)}^2\leq CR \left\|\frac1h\eps_h(u)\right\|_{L^2(Q_{T'})}^2 }\\
    && + C_0\left( \|w_0\|_{L^2(\Omega)}^2+\|w_1\|_{W'_h(\Omega)}^2 + \|f_1\|_{L^1(0,T;L^2_h)}^2+ \|f_2\|_{L^1(0,T;L^2)}^2\right)
\end{eqnarray*}
provided that $RT'\leq \kappa$. By the lemma of Gronwall we obtain (\ref{eq:WeakLinEstim}) for all $0<T<\infty$ such that $RT\leq \kappa$. Now, if $0<T<\infty$ with $RT> \kappa$, we apply the latter estimate successively for some $0=T_0<T_1<\ldots< T_N=T$ such that $R (T_{j+1}-T_j) \leq \kappa$, $j=0,\ldots,N-1$, and $N \leq 2 R\kappa^{-1}T$.  Hence we obtain
  \begin{eqnarray*}
    \lefteqn{\left\|\left(w, \frac1h \eps_h (u)\right)\right\|_{C([0,T];L^2)}}\\
    &\leq& (C_0)^Ne^{CRT}\left(\|f_1\|_{L^1(0,T;L^2_h)}+\|f_2\|_{L^1(0,T;L^2)} +
      \left\|w_0\right\|_{L^2} + \|w_1\|_{W'_h(\Omega)}\right),
  \end{eqnarray*}
  where 
  \begin{equation*}
    (C_0)^N\leq \exp \left(2\kappa^{-1}\ln C_0 RT \right)\leq \exp(C'_0RT)
  \end{equation*}
  since $N \leq 2 R\kappa^{-1}T$. This implies (\ref{eq:WeakLinEstim}) for some modified $C_0,C$ independent of $R\in (0,R_0]$, $h\in (0,1]$, $0<T<\infty$.
\end{proof}

\subsection{Uniform bounds and Proof of Theorem~\ref{thm:main1}}\label{sec:uniform}


For the following we assume that $\theta\geq 0$, $0<T\leq 1$, and  $u_{j,h}, j=0,\ldots,4, f_h$ are as in Theorem~\ref{thm:main1}. Moreover, we assume that $R_0\in (0,1]$ is so small that all the statements in Section~\ref{sec:Linear} are applicable. -- Note that $T\leq 1$ is not a restriction for the proof of Theorem~\ref{thm:main1}. By a simple scaling with $T^{-1}$ in time $t$ and $h$ we can always reduce to this case changing $M>0$ by a  certain factor depending on $T$ if necessary. (Of course this finally influences the smallness assumption of $h_0>0$ in the case $\theta>0$ and the smallness assumption on $M$ if $\theta=0$.) 

Moreover,
let $C_L\geq 1$ be the constant in Lemma~\ref{lem:LinEstim3}. Then (\ref{eq:AssumInitialData1})-(\ref{eq:AssumFData})  imply 
\begin{eqnarray}\nonumber
 \lefteqn{ \max_{|\gamma|\leq 2}\left\|h^{1+\theta}\partial_z^\gamma f_h\right\|_{W^1_1([0,T];L^2)}+ \left\|\frac1h \eps_h(u_{0,h})\right\|_{H^1(\Omega)} }\\\label{eq:DataEstim}
 && + \max_{k=0,1,2}\left\|\left(\frac1h \eps_h (u_{k+1,h}),u_{k+2,h}\right)\right\|_{H^{2-k,0}(\Omega)}\leq \tilde{M} h^{1+\theta}\qquad \qquad\quad
\end{eqnarray}
where $\tilde{M}= C_0M$ for some universal constant $C_0\geq 1$.
If $\theta>0$, we can find some $h_0\in (0,1]$ (depending on $M$) such that $R:=6C_L \tilde{M} h_0^{\theta}\leq R_0$. 
If
$\theta=0$, we assume that $M>0$ is so small that $R:=6C_L \tilde{M}\leq R_0$. 
In this case we set $h_0=1$.

Let $u_h$ be the solution of  (\ref{eq:Evol1})-(\ref{eq:BC3}) due Theorem~\ref{thm:ShortTimeExistence}. 

Since $u_h\in C^4([0,T_{max}(h));L^2)\cap C^0([0,T_{max}(h));H^4)$, there is some $T'=T'(h)\in (0,T_{max}(h))$ such that 
 \begin{equation}\label{eq:BddUn} 
 \max_{|\gamma|\leq 2} \left\|\left(\partial_t^2\partial_z^\gamma u_h, \frac1h\eps_h(\partial_z^\gamma u_h),
        \nabla_{x,t}\frac1h\eps_h(\partial_z^\gamma u_h)\right)\right\|_{C([0,T'];L^2)}\leq 3C_L\tilde{M}h^{1+\theta}
  \end{equation}
where $3C_L\tilde{M}h^\theta <R_0$.
Hence $u_h$ satisfies \eqref{eq:UniformEstim} and we can appy Lemma~\ref{lem:LinEstim3}.

To this end we use that $w^j_h= \partial_{z_j}u_h$, $j=0,\ldots, d-1$, solves
\begin{eqnarray}\label{eq:Approx1j}
  \partial_t^2 w^j_h - \frac1{h^2} \Div_h D^2\widetilde{W}(\nabla_h u)\nabla_h w^j_h
  &=& \partial_{z_j} f_h h^{1+\theta} \quad \text{in}\ \Omega\times (0,T')\\\label{eq:Approx2j}
  \left. D^2\widetilde{W}(\nabla_h u_h)\nabla_h w^j_h e_d\right|_{x_d=\pm \frac12} &=&0, 
   \\\label{eq:Approx3j}
 w^j_h &\text{is} & 2L\text{-periodic in}\ x_j, j=1,d-1,\\\label{eq:Approx4j}
  \left. (w^j_h,\partial_t w^j_h)\right|_{t=0} &=& (w_{0,h}^j, w_{1,h}^j)
\end{eqnarray}
with $w^j_{k,h}= \partial_{x_j} u_{k,h}$, $k=0,1$ if $j=1,\ldots,d-1$ and $w^0_{k,h}= u_{k+1,h}$.
Hence applying Lemma~\ref{lem:LinEstim3} we obtain
 \begin{equation*}
  \max_{|\gamma|= 1,2} \left\|\left(\partial_t^2\partial_{z}^\gamma u_h,\partial_{z}^\gamma\frac1h\eps_h(u_h),
       \nabla_{x,t}\partial_{z}^\gamma\frac1h\eps_h(u_h)\right)\right\|_{C([0,T'];L^2)}\leq 2C_Le^{C'\tilde{M}h^\theta}\tilde{M}h^{1+\theta}
  \end{equation*}
  uniformly in $0<h\leq h_0$.  Due to (\ref{eq:DataEstim}) and
\begin{equation*}
  \nabla_hu_h =  \nabla_h u_{0,h} + \int_0^t \nabla_h w_h^0(\tau) \sd \tau,  
\end{equation*}
 we conclude
 \begin{equation*}
  \max_{|\gamma|\leq 2} \left\|\left(\partial_t^2\partial_{z}^\gamma u_h,\partial_{z}^\gamma\frac1h\eps_h(u_h),
       \nabla_{x,t}\partial_{z}^\gamma\frac1h\eps_h(u_h)\right)\right\|_{C([0,T'];L^2)}\leq 4C_Le^{C'\tilde{M}h^\theta}\tilde{M}h^{1+\theta}.
  \end{equation*}
If $\theta>0$, we can now choose $0<h_0\leq 1$ so small that
 \begin{equation}\label{eq:Estim5}
  \max_{|\gamma|\leq 2} \left\|\left(\partial_t^2\partial_{z}^\gamma u_h,\partial_{z}^\gamma\frac1h\eps_h(u_h),
       \nabla_{x,t}\partial_z^\gamma\frac1h\eps_h(u_h)\right)\right\|_{C([0,T'];L^2)}\leq 5C_L\tilde{M}h^{1+\theta} 
  \end{equation}
  uniformly in $0<h\leq h_0$ where $5C_L\tilde{M}h_0^{\theta}\leq R_0 $. If $\theta=0$, then we choose $\tilde{M}=C_0M$
  sufficiently small to obtain the same estimates. Since $u_h\in C^4([0,T_{max}(h));L^2)\cap C^0([0,T_{max}(h));H^4)$, we can repeat the estimates above and conclude that \eqref{eq:Estim5} holds for $T'=\min (1, T'')$ for any $0<T''<T_{max}(h)$. In particular this shows that $u_h $ cannot leave $U_h$, where $U_h$ is as in Remark~\ref{rem:Neighborhood}, and 
  \begin{equation*}
    \int_0^{T''} \|\nabla_{x,t} u_h(t)\|_{L^\infty(\Omega)}\sd t <\infty.
  \end{equation*}
 Hence the characterization of $T_{max}(h)$ in Theorem~\ref{thm:ShortTimeExistence} shows that  we can choose $T'=T$ in \eqref{eq:Estim5}. Therefore Theorem~\ref{thm:main1} is proved.

\section{First Order Asymptotics}\label{sec:Asymptotics}

Throughout this section we assume that
$$
f_h(x,t) = 
\begin{pmatrix}
  0 \\ g(x',t)
\end{pmatrix},
$$
for some given  $g\in \bigcap_{j=0}^3W^j_1(0,T;H^{10-2j}_{per}((-L,L)^{d-1}))$ independent of $h$. For simplicity let $W(F)= \dist(F,SO(d))^2$, which implies $D^2W(0)F= \sym F$. 
 As seen in the proof of Lemma~\ref{lem:LinEstim3}, 
we can assume without loss of generality that $\int_{(-L,L)^{d-1}} g(x',t) dx'=0$ for all $t\in [0,T]$. Moreover, we assume that $0<\theta \leq 1$.

In this section we construct an approximate solution to the $d$-dimensional system (\ref{eq:Evol1})-(\ref{eq:BC3}) with the aid of a solution to a $(d-1)$-wave equation. 
The ansatz for such an approximate solution is
\begin{eqnarray*}
  \tilde{u}_h(x,t) &=&
  h^{\theta}\begin{pmatrix}
    0\\ hv(x',t)
  \end{pmatrix}
  +
 h^{2+\theta}\begin{pmatrix}
    -x_d \nabla_{x'} v(x',t)\\ 0
  \end{pmatrix} + O(h^{3+\theta}).
\end{eqnarray*}
Then 
\begin{eqnarray*}
  \eps_h(\tilde{u}_h(x,t)) &=&
  \begin{pmatrix}
    h^{2+\theta}x_d\nabla_{x'} v(x',t) & 0 \\ 0& 0
  \end{pmatrix}
\end{eqnarray*}
and therefore
\begin{equation*}
  E^h(\Id+\tilde{u}_h(t))= \int_\Omega \left(\widetilde{W}(\nabla_h \tilde{u}_h(x,t))- h^2g_h(x',t)h^{1+\theta} v(x',t) \right) \sd x = O(h^{4+2\theta})
\end{equation*}
since $\tilde{f}_h=h^\theta g$, cf. Introduction. In order to get a solution of (\ref{eq:Evol1})-(\ref{eq:BC3}), where \eqref{eq:Evol1} is solved in highest order, suitable higher order corrections have to be adapted and $v$ will be determined by a $(d-1)$-dimensional wave equation. Moreover, we will determine suitable ``well prepared initial data'' $(u_{0,h},u_{1,h})$ (independence of the initial data for $v$) such that Theorem~\ref{thm:main1} is applicable and yields a solution $u_h$ of (\ref{eq:Evol1})-(\ref{eq:BC3}). Then we will be able to show that $u_h-\tilde{u}_h$ is of order $O(h^{1+2\theta})$.

More precisely:
Let  $v$ be the solution of the $(d-1)$-dimensional wave equation
\begin{alignat*}{2}
  \partial_t^2 v + \frac1{12}\Delta^2_{x'} v &= g &\quad & \text{in}\  (-L,L)^{d-1}\times (0,T),\\
v &\ \text{is}\ 2L\text{-periodic} &&\text{in}\ x_j,  j=1,\ldots,d-1,\\
(v, \partial_t v)|_{t=0}&= (v_0, v_1) &&\text{in}\  (-L,L)^{d-1},
\end{alignat*}
where $v_0 \in H^{12}_{per}((-L,L)^{d-1}), v_1 \in H^{10}_{per}((-L,L)^{d-1})$.
By standard methods the latter system possesses a unique solution 
\begin{alignat*}{1}
v \in &\, \bigcap_{j=0}^4C^j([0,T];H^{12-2j}_{per}((-L,L)^{d-1}))
\end{alignat*}
Using $v$, we define an approximate solution
$\tilde{u}_h$ of (\ref{eq:Evol1})-(\ref{eq:BC3}) by
\begin{eqnarray*}
  \tilde{u}_h(x,t) &=&
  h^{\theta}\begin{pmatrix}
    0\\ hv
  \end{pmatrix}
  +
 h^{2+\theta}\begin{pmatrix}
    -x_d \nabla_{x'} v\\ 0
  \end{pmatrix}
  +
 h^{4+\theta}\begin{pmatrix}
   (\frac13x_d^3-\frac14x_d) \nabla_{x'} \Delta_{x'}v\\ 0
  \end{pmatrix}\\
&&  +
 h^{5+\theta}\begin{pmatrix}
   0\\(\frac1{48}x_d^2-\frac1{24}x_d^4-\frac1{24\cdot 16}) \Delta_{x'}^2v
  \end{pmatrix}.
\end{eqnarray*}
Then
\begin{eqnarray*}
  \nabla_h  \tilde{u}_h &=& h^{1+\theta}
  \begin{pmatrix}
    0 &-\nabla_{x'} v\\
    \nabla_{x'}v^T & 0
  \end{pmatrix}
  +
h^{2+\theta}
  \begin{pmatrix}
    -x_d\nabla_{x'}^2 v &h(x_d^2-\frac14)\nabla_{x'} \Delta_{x'}v\\
    0 & 0
  \end{pmatrix}\\
&&+h^{4+\theta}
  \begin{pmatrix}
     (\frac13x_d^3-\frac14x_d) \nabla_{x'}^2 \Delta_{x'}v&0\\
    0 & (\frac1{24}x_d-\frac1{6}x_d^3) \Delta_{x'}^2v
  \end{pmatrix}\\
&&+h^{5+\theta}
  \begin{pmatrix}
     0 &0\\
    (\frac1{48}x_d^2-\frac1{24}x_d^4-\frac1{24\cdot 16}) \nabla_{x'}^T\Delta_{x'}^2v &0 
  \end{pmatrix}.
\end{eqnarray*}
Hence
\begin{eqnarray*}
 \lefteqn{\eps_h  (\tilde{u}_h) = 
h^{2+\theta}
  \begin{pmatrix}
    -x_d\nabla_{x'}^2 v &\frac{h}2(x_d^2-\frac14)\nabla_{x'} \Delta_{x'}v\\
    \frac{h}2(x_d^2-\frac14)\nabla_{x'}^T \Delta_{x'}v & h^2 (\frac1{24}x_d-\frac1{6}x_d^3) \Delta_{x'}^2v
  \end{pmatrix}}\\
&&+h^{4+\theta}
  \begin{pmatrix}
     (\frac13x_d^3-\frac14x_d) \nabla_{x'}^2 \Delta_{x'}v&\frac{h}2(\frac1{48}x_d^2-\frac1{24}x_d^4-\frac1{24\cdot 16}) \nabla_{x'}\Delta_{x'}^2v\\
    \frac{h}2(\frac1{48}x_d^2-\frac1{24}x_d^4-\frac1{24\cdot 16}) \nabla_{x'}^T\Delta_{x'}^2v & 0
  \end{pmatrix}
\end{eqnarray*}
and therefore
\begin{equation}\label{eq:BCApproxSol}
  \frac1h \eps_h(\tilde{u}_h)e_d|_{x_d=\pm\frac12} =0.
\end{equation}
Moreover,
\begin{eqnarray*}
  \frac1{h^2}\Div_h \eps_h  (\tilde{u}_h) &=& 
h^{\theta}
  \begin{pmatrix}
    -x_d\nabla_{x'}\Delta_{x'} v +x_d\nabla_{x'} \Delta_{x'}v\\
    \frac{h}2(x_d^2-\frac14)\Delta_{x'}^2v + h (\frac1{24}-\frac1{2}x_d^2) \Delta_{x'}^2v
  \end{pmatrix}\\
&&+h^{2+\theta}
  \begin{pmatrix}
     (\frac13x_d^3-\frac14x_d) \nabla_{x'} \Delta_{x'}^2v+\frac{1}2(\frac1{24}x_d-\frac1{6}x_d^3) \nabla_{x'}\Delta_{x'}^2v\\
    \frac{h}2(\frac1{48}x_d^2-\frac1{24}x_d^4-\frac1{24\cdot 16})
    \Delta_{x'}^3v 
  \end{pmatrix}
\\
&\equiv&
h^{1+\theta}
  \begin{pmatrix}
    0\\
    -\frac1{12}\Delta_{x'}^2v 
  \end{pmatrix}+\tilde{r}_h,
\end{eqnarray*}
where 
\begin{equation*}
  \|\tilde{r}_h\|_{C^2([0,T];L^2(\Omega))}\leq Ch^{2+\theta}.
\end{equation*}
Thus $\tilde{u}_h$ is a solution of 
\begin{eqnarray}\label{eq:ApproxEvol1}
  \partial_t^2 \tilde{u}_h - \frac1{h^2} \Div_h\left( D^2\widetilde{W}(0)\nabla_h \tilde{u}_h\right) &=& f_h h^{1+\theta} -r_h \qquad \text{in}\ \Omega\times (0,T),\\\nonumber
  \left. \left(D^2\widetilde{W}(0)\nabla_h \tilde{u}_h\right) e_d\right|_{x_d=\pm \frac12} &=& 0, \\\nonumber
\tilde{u}_h & \text{is}& 2L\text{-periodic in}\ x_j,  j=1,\ldots,d-1,\\\nonumber
  \left. (\tilde{u}_h,\partial_t \tilde{u}_h)\right|_{t=0} &=& (\tilde{u}_{0,h}, \tilde{u}_{1,h}),
\end{eqnarray}
where
\begin{eqnarray*}
  \tilde{u}_{j,h}(x) &=&
  h^{1+\theta}\begin{pmatrix}
    0\\ v_j
  \end{pmatrix}
  +
 h^{2+\theta}\begin{pmatrix}
    -x_d \nabla_{x'} v_j\\ 0
  \end{pmatrix}
  +
 h^{4+\theta}\begin{pmatrix}
   (\frac13x_d^3-\frac14x_d) \nabla_{x'} \Delta_{x'}v_j\\ 0
  \end{pmatrix}\\
&&  +
 h^{5+\theta}\begin{pmatrix}
   0\\(\frac1{48}x_d^2-\frac1{24}x_d^4-\frac1{24\cdot 16}) \Delta_{x'}^2v_j
  \end{pmatrix}, \quad j=0,1,2,
\end{eqnarray*}
$v_j= \partial_t^jv|_{t=0}$, and
\begin{equation}\label{eq:rhEstim}
  \|r_h\|_{C^2([0,T];L^2(\Omega))}\leq Ch^{2+\theta}.
\end{equation}

We will compare this approximate solution  with the exact solution of the $d$-dimensional system (\ref{eq:Evol1})-(\ref{eq:BC3}) for an appropriate choice of initial values. 
\begin{theorem}\label{thm:FirstOrderAsymp}
  Let $0<\theta \leq 1$, let $v_0,v_1,f_h$, $\tilde{u}_{0,h}, \tilde{u}_{1,h}$ and $\tilde{u}_h$ be defined as above. 
Then for some sufficiently small $h_0\in (0,1]$ and $h\in (0,h_0]$ there are initial values $(u_{0,h},u_{1,h})$ satisfying (\ref{eq:AssumInitialData1}) and such that
\begin{equation*}
  \max_{j=0,1,2}\left\|\frac1h \eps_h(u_{j,h}) -\frac1h \eps_h(\tilde{u}_{j,h})\right\|_{L^2(\Omega)} \leq C h^{1+2\theta}.
\end{equation*}
Moreover, if $u_h$ is the solution of (\ref{eq:Evol1})-(\ref{eq:BC3}), whose existence is assured by Theorem~\ref{thm:main1}, then
\begin{eqnarray*}
 \left\|\left(\partial_t (u_h - \tilde{u}_h), \frac1h \eps_h(u_h-\tilde{u}_h) \right)\right\|_{L^\infty(0,T;L^2)}\leq Ch^{1+2\theta}
\end{eqnarray*}
for all $0<h\leq h_0$ and some $C>0$ independent of $h$.
\end{theorem}
\begin{proof}
We construct the initial values 
$(u_{0,h}, u_{1,h})$ such that $(u_{0,h}, u_{1,h}, u_{2,h})$ solve the system
\begin{eqnarray}\label{eq:EqIV1}
  \frac1{h^2}(DW(\nabla_h u_{0,h}),\nabla_h \varphi)_\Omega &=&    (h^{1+\theta}f_h|_{t=0},\varphi)_\Omega- (u_{2,h},\varphi)_\Omega,\\\label{eq:EqIV2}
  \frac1{h^2}(D^2W(\nabla_h u_{0,h})\nabla_h u_{1,h},\nabla_h \varphi)_\Omega &=&   (h^{1+\theta}\partial_t f_h|_{t=0},\varphi)_\Omega- (u_{3,h},\varphi)_\Omega 
\end{eqnarray}
and
\begin{eqnarray}\nonumber
  \lefteqn{\frac1{h^2}(D^2W(\nabla_h u_{0,h})\nabla_h u_{2,h},\nabla_h \varphi)_\Omega=  \left(h^{1+\theta}\partial_t^2 f_h|_{t=0}-u_{4,h},\varphi\right)_\Omega}\\\label{eq:EqIV3}
 &&  - \frac1{h^2}\left(D^3W(\nabla_h u_{0,h})(\nabla_h u_{1,h},\nabla_h u_{1,h}),\nabla_h \varphi\right)_\Omega 
\end{eqnarray}
for all $\varphi\in H^1_{per}(\Omega)^d$, where
\begin{equation}\label{eq:ast2}
  u_{2+j,h} = h^{1+\theta}
  \begin{pmatrix}
    0 \\  v_{2+j}
  \end{pmatrix}
  +h^{2+\theta}
  \begin{pmatrix}
    -x_d \nabla_{x'} v_{2+j}\\ 0  
  \end{pmatrix}, \quad j=1,2,
\end{equation}
and $v_{2+j}= \partial_{t}^{2+j} v|_{t=0}$.
Hence $\int_\Omega u_{2+j,h} \sd x=0$ for $j=1,2$ and
\begin{equation*}
  \frac1h \eps_h(u_{3,h}) = h^{1+\theta}
  \begin{pmatrix}
    -x_d\nabla_{x'}^2 v_{3} & 0 \\
    0 & 0
  \end{pmatrix}.
\end{equation*}
In particular, this implies
\begin{equation*}
  \left\|\left(u_{4,h},\frac1h \eps_h(u_{3,h}),h^{1+\theta}\partial_t^2 f_h|_{t=0}\right)\right\|_{L^2}+ \left\|\left(u_{3,h}, h^{1+\theta}\partial_t f|_{t=0}\right)\right\|_{H^{1,0}}\leq Ch^{1+\theta}, 
\end{equation*}
where we note that $f$ is independent of $x_d$.
 Because of Proposition~\ref{prop:SolvIV} below, $(u_{0,h},u_{1,h},u_{2,h})$ exist for all $0<h\leq h_0$ if $h_0\in (0,1]$ is sufficiently small and satisfy (\ref{eq:AssumInitialData1})  and 
 \begin{eqnarray*}
\max_{j=0,1,k=0,1,2}\left\|\left(\nabla^j\frac1h \eps_h (u_{k,h}),\nabla_h^2 u_{k,h}\right)\right\|_{H^{2-k,0}} &\leq& Ch^{1+\theta},
 \end{eqnarray*}
In particular, $u_{j,h}$, $j=0,\ldots,4$ satisfy
(\ref{eq:AssumInitialData1}) and (\ref{eq:u2h})-(\ref{eq:u4h}).   
Moreover, we have that 
\begin{equation}\label{eq:ast}
  \max_{j=0,1,2}\left\|\frac1h \eps_h(u_{j,h}) -\frac1h \eps_h(\tilde{u}_{j,h})\right\|_{L^2(\Omega)} \leq C h^{1+2\theta}
\end{equation}
because of Proposition~\ref{prop:SolvIV} below again.

Now let $u_h$ be the solution of (\ref{eq:Evol1})-(\ref{eq:BC3}) due to Theorem~\ref{thm:main1} and consider $w_h= \partial_t u_h - \partial_t \tilde{u}_h$. 
Then $w_h$ solves
\begin{eqnarray*}
  -(\partial_t w_h,\partial_t \varphi)_{Q_T} &+& \frac1{h^2} (D^2\widetilde{W}(\nabla_h
  u_h)\nabla_h w_h,\nabla_h \varphi)_{Q_T} - (w_{1,h},\varphi|_{t=0})_\Omega\\
 &=& -\frac1{h^2}\left((D^2\widetilde{W}(\nabla_h u_h)-D^2\widetilde{W}(0))\nabla_h \partial_t \tilde{u}_h, \nabla_h \varphi \right)_{Q_T}- (\partial_t r_h,\varphi)_{Q_T}  \\
 w_h & \text{is}& 2L\text{-periodic w.r.t.}\ x_j, j=1,\ldots,d-1,\\
  w_h|_{t=0} &=& w_{0,h}.
\end{eqnarray*}
for all $\varphi \in C^1([0,T];H^1_{per}(\Omega)^d)$ with $\varphi|_{t=T}=0$, where $w_{j,h}= u_{1+j,h} -  \tilde{u}_{1+j,h}$, $j=0,1$, and $r_h$ satisfies (\ref{eq:rhEstim}).
Moreover, 
\begin{eqnarray*}
  \lefteqn{\left|\frac1{h^2}\left((D^2\widetilde{W}(\nabla_h u_h)-D^2\widetilde{W}(0))\nabla_h \partial_t \tilde{u}_h, \nabla_h \varphi \right)_{\Omega}\right|}\\
 &\leq& \frac{C}h \left\|\frac1h \eps_h (u_h)\right\|_{L^\infty(0,T;L^2)}\left\|\frac1h \eps_h (\partial_t \tilde{u}_h)\right\|_{L^\infty(0,T;V)} \left\|\frac1h \eps_h (\varphi)\right\|_{L^2(\Omega)}\\
&\leq & Ch^{1+2\theta}\left\|\frac1h \eps_h (\varphi)\right\|_{L^2(\Omega)}
\end{eqnarray*}
due to a similar estimate as in (\ref{eq:EstimD3W}). Hence Lemma~\ref{lem:LinEstimWeak} implies
\begin{eqnarray*}
 \left\|\left(\partial_t (u_h - \tilde{u}_h), \frac1h \eps_h(u_h-\tilde{u}_h) \right)\right\|_{L^\infty(0,T;L^2)}\leq Ch^{1+2\theta}
\end{eqnarray*}
since $\|w_{j,h}\|_{L^2}=O(h^{1+2\theta})$ for $j=0,1$.
This proves the theorem.
\end{proof}
\begin{proposition}\label{prop:SolvIV}
  Let $0<\theta \leq 1$, let $\tilde{u}_h$ be defined as above, $\tilde{u}_{j,h}=\partial_t^j\tilde{u}_h|_{t=0}$, $j=0,1,2$, and $u_{3,h},u_{4,h}$ be as in \eqref{eq:ast2}. 
Then for some sufficiently small $h_0\in (0,1]$ there are initial values $(u_{0,h},u_{1,h},u_{2,h})$ satisfying \eqref{eq:EqIV1}-\eqref{eq:EqIV3} such that
\begin{eqnarray*}
\max_{j=0,1,k=0,1,2}\left\|\left(\nabla^j\frac1h \eps_h (u_{k,h}),\nabla_h^2 u_{k,h}\right)\right\|_{H^{2-k,0}(\Omega)} &\leq& Ch^{1+\theta},\\
  \max_{j=0,1,2}\left\|\frac1h \eps_h(u_{j,h}) -\frac1h \eps_h(\tilde{u}_{j,h})\right\|_{L^2(\Omega)} &\leq& C h^{1+2\theta}.
\end{eqnarray*}
for all $0<h\leq h_0$ and some $C>0$ independent of $h$.
\end{proposition}
In order to prove Propositions~\ref{prop:SolvIV} we have to determine $u_{0,h}$ in dependence of $u_{2,h}$. To this end we will use:
 \begin{proposition}
   \label{prop:SolvIV'}
     Let $0<h\leq 1$. 
  Then there are constants $C_0>0,M_0\in (0,1]$  such that for any $f\in H^{2,0}(\Omega)^d$ with $\|f\|_{H^{2,0}}\leq M_0h$ and $\int_\Omega f\sd x=0$ there is a unique solution $w \in H^{2,2}(\Omega)^d\cap H^{4,0}(\Omega)^d $ with $\int_\Omega w\sd x=0$ such that
  \begin{equation}\label{eq:FnEq}
    \frac1{h^2}\left( D\widetilde{W}(\nabla_h w), \nabla_h\varphi \right)_{L^2(\Omega)} =(f,\varphi)_{L^2(\Omega)}
  \end{equation}
for all $\varphi \in H^1_{per}(\Omega)^d$ and
\begin{equation}\label{eq:FnEstim}
  \left\|\left(\frac1h \eps_h (w), \nabla \frac1h \eps_h (w), \nabla_h^2 w\right)\right\|_{H^{2,0}(\Omega)} \leq C_0\|f\|_{H^{2,0}(\Omega)}.
\end{equation}
for some $C_0>0$ independent of $h,f$. Moreover, if $f'\in H^{2,0}(\Omega)^d$ with $\|f'\|_{H^{2,0}}\leq M_0h$ and $w'\in H^{2,2}(\Omega)^d\cap H^{4,0}(\Omega)^d$ is the solution of \eqref{eq:FnEq}  with $f'$ instead of $f$, then
\begin{eqnarray}\label{eq:FnEstim'}
  \left\|\left(\frac1h \eps_h (w-w'), \nabla \frac1h \eps_h (w-w'), \nabla_h^2 (w-w')\right)\right\|_{H^{2,0}} &\leq& C_0\|f-f'\|_{H^{2,0}}
\end{eqnarray}
for some $C_0>0$ independent of $h,f,f'$.
 \end{proposition}
 \begin{proof}
   First of all (\ref{eq:FnEq}) is equivalent to  
   \begin{eqnarray*}
     \lefteqn{\weight{L_h w, \varphi}_{W'_h,W_h} := \frac1{h^2}\left( D^2\widetilde{W}(0)\nabla_h w, \nabla_h\varphi \right)_{L^2(\Omega)}}\\
 &=&(f,\varphi)_{L^2(\Omega)}- 
\frac1{h^2}\left( G(\nabla_h w), \nabla_h\varphi \right)_{L^2(\Omega)} 
   \end{eqnarray*}
where $G$ is defined by 
\begin{eqnarray}\nonumber
  D\widetilde{W}(\nabla_h u) &=& D^2\widetilde{W}(0)\nabla_h u + 
  \int_0^1D^3\widetilde{W}( \tau\nabla_h u) [\nabla_h u, \nabla_h u](1-\tau)\sd \tau \\\label{eq:DefnF'}
  &\equiv& D^2\widetilde{W}(0)\nabla_h u + 
  G (\nabla_h u). 
\end{eqnarray}
For the following let $G_h(w):= \frac1{h^2} G(\nabla_h w)$.

We will prove the proposition with the aid of the contraction mapping principle. To this end we note that 
 for every $f\in H^{k,0}(\Omega)^d$, $k=0,1$ and $F\in H^{1+k,0}(\Omega)^{d\times d}$ there is a unique $w\in H^1_{per}(\Omega)^d$ with $\nabla w\in H^{k,0}(\Omega)$ such that
   \begin{equation}\label{eq:Lh}
     \weight{L_h w, \varphi}_{W'_h,W_h} = (f,\varphi)_{L^2(\Omega)}+ (F,\nabla_h\varphi)_{L^2(\Omega)}
   \end{equation}
for all $\varphi \in H^1_h(\Omega)$ because of the Lemma of Lax-Milgram, Korn's inequality, and since $L_h$ commutes with tangential derivatives.
The solution satisfies
\begin{equation}\label{eq:C0}
  \left\|\frac1h \eps_h(w)\right\|_{H^{k+1,0}(\Omega)}\leq C_0 \left(\|f\|_{H^{k,0}(\Omega)}+ \|F\|_{H^{k+1,0}_h(\Omega)}\right),\quad k=0,1
\end{equation}
for some universal $C_0>0$. Moreover, if $F\in H^1_{per}(\Omega)^{d\times d}$ with $\nabla F\in H^{k,0}(\Omega)$, then (\ref{eq:Lh}) implies 
\begin{equation*}
  -\frac1{h^2} \Div_h (D^2\widetilde{W}(0)\nabla_h w)= f-\Div_h F \qquad \text{in}\ \mathcal{D}'(\Omega).
\end{equation*}
Therefore $w\in H^{2}(\Omega)^d$ with $\nabla^2 w\in H^{k,0}(\Omega)$ by standard elliptic regularity. Hence Lemma~\ref{lem:LinEstim2} together with the previous estimate imply
\begin{eqnarray}\nonumber
  \lefteqn{\left\|\left(\frac1h \eps_h(w),\nabla \frac1h \eps_h(w),\nabla_h^2 w\right)\right\|_{H^{k,0}(\Omega)}}\\\label{eq:LhEstim}
&\leq& C_0 \left(\left\|\left(f, h^2\nabla_h F\right)\right\|_{H^{k,0}(\Omega)}+ \|F\|_{H^{1+k,0}_h(\Omega)}\right)
\end{eqnarray}
for $k=0,1$ and some universal $C_0>0$. Using  estimates based on Corollary~\ref{cor:D3W}, which are similar to the ones in Lemma~\ref{lem:DifEstimAn},  one derives 
\begin{equation*}
  \|G_h(w_1)-G_h(w_2)\|_{H^{2,0}_h(\Omega)} \leq C M_0 \left\|\frac1h\eps_h(w_1-w_2)\right\|_{V(\Omega)}
\end{equation*}
for some $C>0$ provided that 
\begin{equation}\label{eq:Estimwj}
\max_{j=1,2}  \left\|\left(\frac1h \eps_h(w_j),\nabla \frac1h \eps_h(w_j),\nabla_h^2 w_j\right)\right\|_{H^{1,0}(\Omega)}\leq 2C_0M_0 h,
\end{equation}
where $C_0>0$ is as (\ref{eq:C0}) and $M_0\in (0,1]$. Here we note that
\begin{eqnarray*}
  \partial_{x_k} G_h(w_j)&=& -\frac1{h^2}DG(\nabla_h w_j)\nabla_h \partial_{x_k}w_j,\\
  \partial_{x_k}\partial_{x_l} G_h(w_j)&=& -\frac1{h^2}DG(\nabla_h w_j)\nabla_h \partial_{x_k}\partial_{x_l}w_j\\
&& -\frac1{h^2}  D^3\widetilde{W}( \nabla_h w_j)[\nabla_h \partial_{x_k} w_j, \nabla_h \partial_{x_l} w_j]
\end{eqnarray*}
for all $k,l=1,\ldots, d-1$, $j=1,2$,
where $DG(\nabla_h w_j)=D^2\widetilde{W}(\nabla_h w_j) - D^2 \widetilde{W}(0)$. To estimate the $DG$-terms one uses (\ref{eq:D3WEstim}) (which yields estimate similiar to (\ref{eq:EstimD3W})) and to estimate the $D^3 \widetilde{W}$-term one uses (\ref{eq:D3WEstimb}).

Furthermore, using Corollary~\ref{cor:Algebra}, one shows in the same way as in the proof of Lemma~\ref{lem:LinEstim2}, 
that
\begin{equation*}
  h^2\|\nabla_h(G_h(w_1)-G_h(w_2))\|_{H^{1,0}(\Omega)} \leq C M_0 \left\|\left(\nabla_h^2(w_1-w_2), \nabla_h (w_1-w_2)\right)\right\|_{H^{1,0}}
\end{equation*}
for some $C>0$ provided that (\ref{eq:Estimwj}) holds. Hence, if $M_0\in (0,1]$ is sufficiently small, we obtain that $L_h^{-1}G_h\colon X_h\to X_h$ restricted to $\ol{B_{2C_0M_0h}(0)}$ is a contraction, where $X_h$ is normed by
\begin{equation*}
   \|w\|_{X_h}:= \left\|\left(\frac1h \eps_h(w),\nabla \frac1h \eps_h(w),\nabla_h^2 w\right)\right\|_{H^{1,0}(\Omega)}.
\end{equation*}
Therefore we obtain a unique solution $w$ solving (\ref{eq:FnEq}) and satisfying (\ref{eq:FnEstim}) and \eqref{eq:FnEstim'} with $H^{2,0}(\Omega)$ replaced by $H^{1,0}(\Omega)$. In order to obtain (\ref{eq:FnEstim}) and \eqref{eq:FnEstim'}, one can simply use that $w_j:= \partial_{x_j}w$, $j=1,\ldots, d-1$, solves
\begin{equation*}
    \frac1{h^2}\left( D^2\widetilde{W}(\nabla_h w)\nabla_h w_j, \nabla_h\varphi \right)_{L^2(\Omega)} =(\partial_{x_j}f,\varphi)_{L^2(\Omega)} \qquad \text{for all}\ \varphi \in H^1_{per}(\Omega)
\end{equation*}
 and apply Lemma~\ref{lem:LinEstim2}.
 \end{proof} 

\noindent
\begin{proof*}{of Proposition~\ref{prop:SolvIV}}
Let $L_h, X_h$ be as in the proof of Proposition~\ref{prop:SolvIV'}.

First of all, \eqref{eq:EqIV1}-\eqref{eq:EqIV3} are equivalent to 
\begin{eqnarray}\label{eq:EqIV2b}
  \frac1{h^2}(D^2W(\nabla_h u_{0,h})\nabla_h u_{1,h},\nabla_h \varphi)_\Omega &=&   (h^{1+\theta}\partial_t f_h|_{t=0},\varphi)_\Omega-(u_{3,h},\varphi)_\Omega  
\end{eqnarray}
and
\begin{eqnarray}\nonumber
  \lefteqn{\frac1{h^2}(D^2W(\nabla_h u_{0,h})\nabla_h u_{2,h},\nabla_h \varphi)_\Omega=  (h^{1+\theta}\partial_t^2 f_h|_{t=0},\varphi)_\Omega }\\\label{eq:EqIV3b}
 &&  - (u_{4,h},\varphi)_\Omega-\frac1{h^2}\left(D^3W(\nabla_h u_{0,h})[\nabla_h u_{1,h},\nabla_h u_{1,h}],\nabla_h \varphi\right)_\Omega 
\end{eqnarray}
for all $\varphi\in H^1_{per}(\Omega)^d$, where $u_{0,h}=G_1(u_{2,h})$ is the solution of \eqref{eq:FnEq} with $f=h^{1+\theta}f_h|_{t=0}-u_{2,h}$. Moreover, because of \eqref{eq:FnEstim'},
\begin{equation}\label{eq:G1Estim}
 \max_{|\gamma|\leq 1}\|\partial_{x'}^\gamma (G_1(u_{2,h})-G_1(u_{2,h}'))\|_{X_h}\leq C_0  \|u_{2,h}- u_{2,h}'\|_{H^{2,0}}
\end{equation}
for all $u_{2,h},u_{2,h}'\in H^{2,0}(\Omega)^d$ with norms bounded by $\frac12M_0h$ and $\|h^{1+\theta} f_h\|_{H^{2,0}(\Omega)}\leq \frac12M_0 h$. Note that the last condition is satisfied for all $0<h\leq 1$ sufficiently small if $u_{2,h}, u_{2,h}'$ are of order $h^{1+\theta}$ in the corresponding spaces.

Hence \eqref{eq:EqIV2b}-\eqref{eq:EqIV3b} are equivalent to
\begin{eqnarray*}
     \lefteqn{\weight{L_h u_{1,h}, \varphi}_{W'_h,W_h} =\left(h^{1+\theta}\partial_t f_h|_{t=0}-u_{3,h},\varphi\right)_{L^2(\Omega)}}\\
 &&- 
\underbrace{\frac1{h^2}\left( DG(\nabla_h u_{0,h})\nabla_h u_{1,h}, \nabla_h\varphi \right)_{L^2(\Omega)}}_{ \equiv (G_2 (u_{1,h},u_{2,h}), \nabla \varphi)_{L^2(\Omega)}},\\
     \lefteqn{\weight{L_h u_{2,h}, \varphi}_{W'_h,W_h} =}\\
 &&\left(h^{1+\theta}\partial_t^2 f_h|_{t=0}-u_{4,h},\varphi\right)_{L^2(\Omega)}- 
\frac1{h^2}\left( DG(\nabla_h u_{0,h})\nabla_h u_{2,h}, \nabla_h\varphi \right)_{L^2(\Omega)}\\
&&- 
\frac1{h^2}\left( D^3\widetilde{W}(\nabla_h u_{0,h})[\nabla_h u_{1,h},\nabla_h u_{1,h}], \nabla_h\varphi \right)_{L^2(\Omega)}\\
&&\equiv \left(h^{1+\theta}\partial_t^2 f_h|_{t=0}-u_{4,h},\varphi\right)_{L^2(\Omega)}+ (G_3 (u_{1,h},u_{2,h}), \nabla \varphi)_{L^2(\Omega)}
   \end{eqnarray*}
for all $\varphi \in H^1_{per}(\Omega)^d$.
As in the proof of Proposition~\ref{prop:SolvIV'} we show the existence of a unique solution with the aid of the contraction mapping principle.

Let us first assume that $(u_{1,h},u_{2,h})$ is a solution of the system above in order to demonstrate the essential estimates.
Then, because of \eqref{eq:LhEstim},  we have that
\begin{eqnarray*}
  \lefteqn{\left\|\left(\frac1h \eps_h(u_{1,h}),\nabla \frac1h \eps_h(u_{1,h}),\nabla_h^2 u_{1,h}\right)\right\|_{H^{1,0}(\Omega)}}\\
&\leq& C_0 \left(\left\|\left(u_{3,h}-h^{1+\theta}\partial_t f, h^2\nabla_h G_2(u_{1,h},u_{2,h})\right)\right\|_{H^{1,0}(\Omega)}+ \|G_2(u_{1,h},u_{2,h})\|_{H^{2,0}_h(\Omega)}\right).
\end{eqnarray*}
Moreover, using Corollary~\ref{cor:D3W} one derives as before 
\begin{eqnarray*}
  \lefteqn{\|G_2(u_{1,h},u_{2,h})-G_2(u_{1,h}',u_{2,h}')\|_{H^{2,0}_h(\Omega)}}\\
 &\leq& C h^{\theta} \left(\left\|\frac1h\eps_h(u_{0,h}-u_{0,h}')\right\|_{V(\Omega)}+ \left\|\frac1h\eps_h(u_{1,h}-u_{1,h}')\right\|_{V(\Omega)}\right) 
\\
 &\leq& C' h^{\theta} \left(\left\|u_{2,h}-u_{2,h}'\right\|_{H^{2,0}(\Omega)}+ \left\|\frac1h\eps_h(u_{1,h}-u_{1,h}')\right\|_{V(\Omega)}\right) 
\end{eqnarray*}
due to \eqref{eq:G1Estim}, and using Corollary~\ref{cor:Algebra} one estimates
\begin{eqnarray*}
 \lefteqn{  h^2\|\nabla_h(G_2(u_{1,h},u_{2,h})-G_2(u_{1,h},u_{2,h}))\|_{H^{1,0}(\Omega)}} \\
&\leq& C h^{\theta} \left\|\left(\nabla_hu_{0,h}, \nabla_hu_{1,h}\right)\right\|_{V_h(\Omega)}\leq C' h^{\theta} \left(\left\|\nabla_hu_{1,h}\right\|_{V_h(\Omega)}+\|u_{2,h}\|_{H^{2,0}(\Omega)} \right). 
\end{eqnarray*}
Furthermore, because of \eqref{eq:LhEstim},
\begin{eqnarray*}
  \lefteqn{\left\|\left(\frac1h \eps_h(u_{2,h}),\nabla \frac1h \eps_h(u_{2,h}),\nabla_h^2 u_{2,h}\right)\right\|_{L^2(\Omega)}}\\
&\leq& C_0 \left(\left\|\left(u_{4,h}-h^{1+\theta}\partial_t^2 f, h^2\nabla_h G_3(u_{1,h},u_{2,h})\right)\right\|_{L^2(\Omega)}+ \|G_3(u_{1,h},u_{2,h})\|_{H^{1,0}_h(\Omega)}\right),
\end{eqnarray*}
where 
\begin{eqnarray*}
  \lefteqn{\|G_3(u_{1,h},u_{2,h})-G_3(u_{1,h}',u_{2,h}')\|_{H^{1,0}_h(\Omega)}}\\
 &\leq& C h^{\theta} \left(\left\|\frac1h\eps_h(u_{0,h}-u_{0,h}')\right\|_{V(\Omega)}+ \left\|\frac1h\eps_h(u_{1,h}-u_{1,h}')\right\|_{V(\Omega)}\right) 
\\
 &\leq& C' h^{\theta} \left(\left\|u_{2,h}-u_{2,h}'\right\|_{H^{2,0}(\Omega)}+ \left\|\frac1h\eps_h(u_{1,h}-u_{1,h}')\right\|_{V(\Omega)}\right) 
\end{eqnarray*}
and 
\begin{eqnarray*}
 \lefteqn{  h^2\|\nabla_h(G_3(u_{1,h},u_{2,h})-G_3(u_{1,h},u_{2,h}))\|_{L^2(\Omega)}} \\
&\leq& C h^{\theta} \left\|\left(\nabla_hu_{0,h}, \nabla_hu_{1,h}\right)\right\|_{V_h(\Omega)}\leq C' h^{\theta} \left(\left\|\nabla_hu_{1,h}\right\|_{V_h(\Omega)}+\|u_{2,h}\|_{H^{2,0}(\Omega)} \right). 
\end{eqnarray*}
Altogether we can write \eqref{eq:EqIV1}-\eqref{eq:EqIV3} as a fixed point equation
\begin{equation*}
  \mathcal{L}_h
  \begin{pmatrix}
    u_{1,h}\\ u_{2,h}
  \end{pmatrix}
  = \mathcal{G}_h(u_{1,h},u_{2,h}),
\end{equation*}
where $\mathcal{L}_h\colon Y_h\to Z_h$ is linear, bounded and invertible, $\mathcal{G}_h\colon \ol{B_R(0)}\subset Y_h\to Z_h$ is Lipschitz continuous with Lipschitz constant of order $h^{1+\theta}$ for all $0<h\leq 1$ sufficiently small, $R>0$, and $Y_h,Z_h$ are Banach spaces normed by
\begin{eqnarray*}
  \|(w_1,w_2)\|_{Y_h} &=& \max_{j=0,1,k=1,2}\left\|\left(\nabla^j\frac1h \eps_h(w_k),\nabla_h^{1+j} w_k\right)\right\|_{H^{2-k}(\Omega)}\\
  \|(g_1,g_2)\|_{Z_h} &=& \max_{j=1,2} \inf_{g_j=f_j+\Div_h F_j}\left(\left\|\left(f_j,h^2 \nabla_h F_j\right)\right\|_{H^{2-j,0}(\Omega)}+ \|F_j\|_{H^{3-j,0}_h(\Omega)}\right),
\end{eqnarray*}
cf. \eqref{eq:LhEstim}.
Hence as in the proof of Proposition~\ref{prop:SolvIV'} one obtains for sufficiently small $0<h\leq 1$ the existence of a unique $(u_{1,h},u_{2,h})\in Y$ solving \eqref{eq:EqIV2b}-\eqref{eq:EqIV3b} such that
\begin{eqnarray*}
  \lefteqn{\max_{j=0,1,k=0,1,2}\left\|\left(\nabla^j\frac1h \eps_h(u_{k,h}),\nabla_h^{1+j} u_{k,h}\right)\right\|_{H^{2-k,0}(\Omega)}}\\
&\leq& C \left(\max_{k=0,1,2}\|\partial_t^k f|_{t=0}\|_{H^{2-k,0}}+ \max_{k=0,1}\|u_{3+k}\|_{H^{1-k,0}}\right).
\end{eqnarray*}
This proves the first part.

Finally, we have that
\begin{eqnarray*}
  \frac1{h^2}(\eps_h (u_{1,h}-\tilde{u}_{1,h}),\eps_h (\varphi))_\Omega &=&  -\frac1{h^2}(DG(\nabla_h u_{0,h})\nabla_h u_{1,h},\nabla_h \varphi)_\Omega  + (r_{1,h},\varphi)_{\Omega} \\
\frac1{h^2}(\eps_h (u_{2,h}-\tilde{u}_{2,h}),\eps_h (\varphi))_\Omega &=&  -\frac1{h^2}(DG(\nabla_h u_{0,h})\nabla_h u_{2,h},\nabla_h \varphi)_\Omega + (r_{2,h},\varphi)_{\Omega} \\
&&- 
\frac1{h^2}\left( D^3\widetilde{W}(\nabla_h u_{0,h})[\nabla_h u_{1,h},\nabla_h u_{1,h}], \nabla_h\varphi \right)_{L^2(\Omega)}
\end{eqnarray*}
for all $\varphi \in H^1_{per}(\Omega)^d$,
where $\max_{j=0,1,2}\|r_{j,h}\|_{L^2(\Omega)}\leq Ch^{1+2\theta}$.
 Here we have used that
 \begin{eqnarray*}
    \frac1{h^2}(\eps_h(\tilde{u}_{j,h}),\eps_h (\varphi))_\Omega &=& -\frac1{h^2}(\Div_hD^2\widetilde{W}(0) \tilde{u}_{j,h},\varphi)_\Omega\\
 &=& h^{1+\theta}\left(\partial_t^j f_h|_{t=0}-\tilde{u}_{2+j,h}, \varphi_d\right)_\Omega+  (\partial_t^jr_h,\varphi)_{\Omega}\\
 &=& h^{1+\theta}\left(\partial_t^j f_h|_{t=0}-u_{2+j,h}, \varphi_d\right)_\Omega+  (r_{j,h},\varphi)_{\Omega}
 \end{eqnarray*}
for $j=1,2$ because of \eqref{eq:ApproxEvol1}, where $\max_{j=1,2}\|\partial_t^jr_{h}\|_{C([0,T];L^2)}\leq Ch^{1+2\theta}$, and $\tilde{u}_{2+j,h}-u_{2+j,h}=O(h^{1+2\theta})$.
Moreover,
\begin{eqnarray*}
  \left|\frac1{h^2}\left((DG(\nabla_h u_{0,h})\nabla_h u_{1,h},\nabla_h \varphi\right)_\Omega\right| \leq Ch^{1+2\theta}\left\|\frac1h\eps_h(\varphi)\right\|_{L^2(\Omega)}\\
  \left|\frac1{h^2}\left(D^3\widetilde{W}(\nabla_h u_{0,h})[\nabla_h u_{1,h},\nabla_h u_{1,h}],\nabla_h \varphi\right)_\Omega\right| \leq Ch^{1+2\theta}\left\|\frac1h\eps_h(\varphi)\right\|_{L^2(\Omega)}
\end{eqnarray*}
for all $\varphi \in H^1_{per}(\Omega)^d$
because of estimates similiar to (\ref{eq:EstimD3W}) and the estimates for $u_{0,h}, u_{1,h}, u_{2,h}$. Hence choosing $\varphi=u_{2,h}-\tilde{u}_{2,h}$ we conclude
\begin{equation*}
  \max_{j=1,2}\left\|\frac1h \eps_h(u_{j,h}) -\frac1h \eps_h(\tilde{u}_{j,h})\right\|_{L^2(\Omega)} \leq C h^{1+2\theta}
\end{equation*}
for all sufficiently small $0<h\leq 1$ due to \eqref{eq:UniformCoercive'}.
Finally, using the estimate for $u_{2,h}-\widetilde{u}_{2,h}$, we have that
\begin{eqnarray*}
  \frac1{h^2}(\eps_h (u_{0,h}-\tilde{u}_{0,h}),\eps_h (\varphi))_\Omega &=&  -\frac1{h^2}(G(\nabla_h u_{0,h}),\nabla_h \varphi)_\Omega + (r_h|_{t=0},\varphi)_{\Omega}
\end{eqnarray*}
for all $\varphi \in H^1_{per}(\Omega)^d$, where $\|r_h\|_{C([0,T];L^2)}\leq Ch^{1+2\theta}$.
Hence using
\begin{eqnarray*}
    \left|\frac1{h^2}\left(G(\nabla_h u_{0,h}),\nabla_h \varphi\right)_\Omega\right| \leq Ch^{1+2\theta}\left\|\frac1h\eps_h(\varphi)\right\|_{L^2(\Omega)}
\end{eqnarray*}
and \eqref{eq:UniformCoercive'} we also obtain
\begin{equation*}
  \left\|\frac1h \eps_h(u_{0,h}) -\frac1h \eps_h(\tilde{u}_{0,h})\right\|_{L^2(\Omega)} \leq C h^{1+2\theta}.
\end{equation*}
\end{proof*}


\appendix
\appendix
\section{Existence of Classical Solutions for fixed $h>0$}

In this appendix we give more detailed comments how the results of \cite{KochWaveEquation} apply to our situation. First of all, in \cite{KochWaveEquation} a quasi-linear hyperbolic system of the form
\begin{alignat}{2}\label{eq:Evo1}
  \sum_{i=0}^d \partial_{x_i} F_j^i(t,x,u,Du) &= w_j(t,x,u,Du)&\quad& \text{in}\ \Omega\times (0,T),\\
  \sum_{i=0}^d \nu_i F_j^i(t,x,u,Du)&= g_j(t,x,u,Du)&\quad& \text{on}\ \partial\Omega\times (0,T),\\\label{eq:Evo3}
(u|_{t=0},\partial_t u|_{t=0})&= (u_0,u_1) && \text{in}\ \Omega
\end{alignat}
is considered, where $j=1,\ldots,N$, $x_0=t$, $\Omega\subseteq \R^d$ is a sufficiently smooth bounded domain, $\nu$ is its outer normal, $u\colon \Omega\times [0,T)\to \R^N$, and $Du$ is the Jacobi matrix of $u$ with respect to $(t,x)$. 

In our situation we do not have a bounded domain. But the equations on $\Omega=(-\tfrac12,\tfrac12)\times (-L,L)^{d-1}$ with periodic tangential boundary conditions are equivalent to the equations on the manifold $\widetilde{\Omega}=(-\frac12,\frac12)\times (\R^{d-1}/2L\Z^{d-1})$, which is a smooth compact manifold with smooth boundary. -- Actually, since the boundary is flat, one can easily differentiate equations with respect to tangential direction (e.g. using the difference quotient method) and obtain standard regularity results for elliptic equations as in the case of a bounded smooth domain. (Proofs even simplify since no localization is needed.) Many arguments in \cite{KochWaveEquation} rely on differentiation in time and applying standard results from elliptic theory, which can be applied the same way if the bounded smooth domain is replaced by $\widetilde{\Omega}$. Therefore all results in \cite{KochWaveEquation} also apply to the case when the bounded domain is replaced by $\widetilde{\Omega}$.
 
To obtain our system \eqref{eq:Evol1}-\eqref{eq:BC1} one simply has to choose $g_j\equiv 0$, $w_j(t,x,u,Du)=-(f_h)_jh^{1+\theta}$, and
\begin{eqnarray*}
  F_j^0(t,x,u,Du)=-u_j,\quad  F_j^i(t,x,u,Du)=(DW(Du))_{j,i}= \frac{\partial W}{\partial(\partial_{i} u_j)}(Du),\quad 
\end{eqnarray*}
for $j=1,\ldots,N=d$, $i=1,\ldots,d$. Then the assumptions 1-5 in \cite{KochWaveEquation} are satisfied: 
Because of
\begin{equation*}
  a_{jl}^{ik}= \frac{\partial F^i_j}{\partial(\partial_{k}u^l)}, \quad i,k=0,\ldots,d, j,l=1,\ldots,d, 
\end{equation*}
$a_{jl}^{ik}=a_{lj}^{ki}$ and the symmetry assumption 2 holds. The coerciveness condition, i.e., assumption 3, is satisfied because of \eqref{eq:coercive} and Korn's inequality. Here we note that we can choose  $\theta=e_0$ (the canonical unit vector in the time direction) as vector field in assumption 3. Then the projection $P$ on $\R^{d+1}$ is simply the projection given by $(t,x)\mapsto x$. Since $a_{jl}^{00}=1$, the assumption 4 is trivial. The assumptions 5 is satisfied because of the compatibility conditions in Theorem~\ref{thm:main1}. Finally, assumption 1 is satisfied with $s=3$ if one would additionally assume $f_h\in C^3(\ol{\Omega}\times[0,T])$. But it is easy to observe from the proof that in the present situation with $(u,Du)$-independent $w_j$ the regularity assumed in Theorem~\ref{thm:main1} is sufficient. In assumption 5 one can e.g. choose $U=(-T,\times T)\times\Omega\times \R^{d}\times \tilde{U}_h$, where 
\begin{equation*}
  \tilde{U}_h=\left\{A\in \R^{d\times d}: \left|\left(A, \tfrac1h\sym A\right)\right|\leq \eps h\right\}
\end{equation*}
 for some sufficiently small $\eps>0$ as in Remark~\ref{rem:Neighborhood}. Moreover, for sufficently small $h>0$, if $\theta>0$, $M>0$, if  $\theta =0 $, respectively, we have that $D_x u_0(x)\in \tilde{U}_h$ for any $x\in \ol{\Omega}$, cf. Section~\ref{sec:Linear}.

Altogether minor modifications of the results and arguments in \cite{KochWaveEquation} show the existence of classical solutions for fixed $h>0$ as stated in Theorem~\ref{thm:ShortTimeExistence}.

\def\cprime{$'$}

\noindent
{\bf Addresses:}\\ Helmut Abels\\
NWF I - Mathematik\\
Universit\"at Regensburg\\
93040 Regensburg, Germany\\
e-mail: helmut.abels@mathematik.uni-regensburg.de\\[1ex]
Maria Giovanna Mora\\
Scuola Internazionale Superiore di Studi Avanzati\\
via Bonomea 265\\
34136 Trieste,Italy \\
e-mail: mora@sissa.it\\[1ex]
Stefan M\"uller\\
Hausdorff Center for Mathematics\\
 Institute for Applied Mathematics\\
Universit\"at Bonn\\
Endenicher Allee 60\\
53115 Bonn, Germany\\
e-mail: stefan.mueller@hcm.uni-bonn.de


\end{document}